\documentclass[12pt]{amsart}
\usepackage[pagewise]{lineno}
\usepackage{graphicx}
\usepackage{mathptmx}

\usepackage{amscd}
\usepackage{amsthm}
\usepackage{amsxtra}
\usepackage{a4wide}
\usepackage{latexsym}
\usepackage{amssymb}
\usepackage{amsfonts}
\usepackage{amsmath}
\usepackage{amsrefs}
\usepackage{mathrsfs}
\usepackage{centernot}

\usepackage{upref}
\usepackage{txfonts}

\usepackage{color,soul}
\usepackage{xcolor}

\usepackage[colorinlistoftodos]{todonotes}

\usepackage[bookmarksnumbered, colorlinks, plainpages]{hyperref}
\hypersetup{colorlinks=true,linkcolor=red, anchorcolor=green, citecolor=cyan, urlcolor=red, filecolor=magenta, pdftoolbar=true}

\allowdisplaybreaks

\theoremstyle{plain}
\newtheorem{thm}{Theorem}[section]
\newtheorem{lem}[thm]{Lemma}

\newtheorem{prop}[thm]{Proposition}
\theoremstyle{definition}
\newtheorem{rem}[thm]{Remark}

\newtheorem{defi}[thm]{Definition}

\newtheorem{conv}[thm]{Convention}

\numberwithin{equation}{section}

\def\esup{\operatornamewithlimits{ess\,sup}}

\def\Id{\operatorname{I}}

\def\ces{\operatorname{Ces}}
\def\cop{\operatorname{Cop}}
\def\Ces{\operatorname{ces}}
\def\Cop{\operatorname{cop}}

\def\ap{\approx}
\def\qq{\qquad}
\def\rw{\rightarrow}

\def\ls{\lesssim}

\def\hra{\hookrightarrow}

\def\M{\mathcal M}
\def\m{\mathcal M}

\def\a{\alpha}
\def\b{\beta}

\def\la{\lambda}

\def\vp{\varphi}

\def\i{\infty}
\def\I{(0,\i)}

\def\N{\mathbb N}

\def\R{\mathbb R}
\def\Z{\mathbb Z}
\def\R{\mathbb R}

\def\M{\mathfrak M}

\def\W{{\mathcal W}}

\def\mp{{\mathfrak M}}

\def\a{\alpha}

\def\b{\beta}

\def\O{\Omega}

\def\la{\lambda}

\def\vp{\varphi}

\def\i{\infty}
\def\I{(0,\i)}

\def\ga{{\gamma}}

\def\dual{\,^{^{\bf c}}\!}

\begin{document}

\title{Pointwise Multipliers between weighted Copson and Ces\`{a}ro function spaces}

\author[]{A.Gogatishvili, R.Ch.~Mustafayev, T.~{\"U}nver}

\thanks{The research of A. Gogatishvili was partly supported by the grants P201-13-14743S of the Grant Agency of the Czech Republic and RVO: 67985840, by Shota Rustaveli National Science Foundation grants no. 31/48 (Operators in some function spaces and their applications in Fourier Analysis) and no. DI/9/5-100/13 (Function spaces, weighted inequalities for integral operators and problems of summability of Fourier series). The research of all authors was partly supported by the joint project between  Academy of Sciences of Czech Republic and The Scientific and Technological Research Council of Turkey}

\subjclass[2010]{Primary 26D10; Secondary 26D15.}

\keywords{Ces\`{a}ro and Copson function spaces, embeddings, iterated Hardy inequalities, weights.}

\begin{abstract}
In this paper the solution of the pointwise multiplier problem between weighted Copson function spaces  $\operatorname{Cop}_{p_1,q_1}(u_1,v_1)$ and weighted Ces\`{a}ro function spaces $\operatorname{Ces}_{p_2,q_2}(u_2,v_2)$ is presented, where $p_1,\,p_2,\,q_1,\,q_2 \in (0,\infty)$, $p_2 \le q_2$ and $u_1,\,u_2,\,v_1,\,v_2$ are weights on $(0,\infty)$.
\end{abstract}

\maketitle

\

\section{Introduction}\label{introduction}

\

Various properties of the different spaces defined by the Ces\'{a}ro operator have been studied extensively in the literature. The Ces\`{a}ro sequence spaces $\Ces_p$ and the Ces\`{a}ro function spaces $\ces_p$ have been introduced by Shiue in 1970 in \cite{shiue1} and \cite{shiue}, respectively. In 1971 Leibowitz proved that $\Ces_1 = \{0\}$ and for $1 < q < p \leq \infty$, $\ell_p$ and $\Ces_q$ sequence spaces are proper subspaces of $\Ces_p$ \cite{Leibowitz}. Jagers has obtained the associate space for $\Ces_p$, $(1 < p < \infty)$ \cite{jagers}. In \cite{syzhanglee}, Sy, Zhang and Lee gave a description of dual spaces of $\ces(p)$ spaces based on Jagers' result. In 1996 different, isomorphic description due to Bennett appeared in \cite{bennett1996}. In \cite[Theorem 21.1]{bennett1996} Bennett observes that the classical Ces\`{a}ro function space and the classical Copson function space coincide for $p > 1$. He also derives estimates for the norms of the corresponding inclusion operators. The same result, with different estimates, is due to Boas \cite{boas1970}, who in fact obtained the integral analogue of the Askey-Boas Theorem \cite[Lemma 6.18]{boas1967} and \cite{askeyboas}. These results generalized in \cite{grosse} using the blocking technique. In \cite{astasmal2009} they investigated dual spaces for $\ces (p)$ for $1 < p < \infty$. Their description can be viewed as being analogous to one given for sequence spaces in \cite{bennett1996} (For more detailed information about history of classical Ces\`{a}ro spaces see recent survey paper \cite{asmalsurvey}). 

In this paper, we will describe the pointwise multiplier spaces of weighted Copson and Ces\'{a}ro function spaces. The weighted Ces\`{a}ro and Copson function spaces are defined in \cite{gmu} as follows:
\begin{defi}\label{defi.2.1}
Let  $0 <p, q \le \infty$, $u \in \mp^+ \I$, $v\in \W\I$. The weighted Ces\`{a}ro and Copson spaces are defined by
\begin{align*}
\ces_{p,q} (u,v) : & = \bigg\{ f \in \mp^+ \I: \|f\|_{\ces_{p,q}(u,v)} : = \big\| \|f\|_{p,v,(0,\cdot)} \big\|_{q,u,\I} < \i \bigg\}, \\
\intertext{and} \cop_{p,q} (u,v) : & = \bigg\{ f \in \mp^+ \I: \|f\|_{\cop_{p,q} (u,v)} : = \big\| \|f\|_{p,v,(\cdot,\i)} \big\|_{q,u,\I} < \i \bigg\},
\end{align*}
respectively.
\end{defi}

\begin{defi}
Let $0<q\leq \i$. We denote by $\O_q$ the set of all functions $u \in \mp^+ \I$ such that
$$
0<\|u\|_{q,(t,\i)}<\i,~~ t>0,
$$
and by $\dual{\O}_q$ the set of all functions $u \in \mp^+ \I$ such that
$$
0<\|u\|_{q,(0,t)}<\i,~~ t>0.
$$
\end{defi}

Let $v \in \W\I$. It is easy to see that $\ces_{p,q} (u,v)$ and $\cop_{p,q} (u,v)$ are quasi-normed vector spaces when $u \in \O_q$ and $u \in \dual{\O}_q$, respectively.

Note that the function spaces $C$ and $D$ defined by Grosse-Erdmann in \cite{grosse} are related with this definition in the following way:
$$
\ces_{p,q}(u,v) = C(p,q,u)_v \qq \mbox{and} \qq \cop_{p,q}(u,v) = D(p,q,u)_v.
$$

We use the notations $\ces_p(u) : = \ces_{1,p}(u,{\bf 1})$ and $\cop_p(u) : = \cop_{1,p}(u,{\bf 1})$. Obviously, $\ces(p) = \ces_p (x^{-1})$ and $\cop(p) = \cop_p (x^{-1})$. In \cite{kamkub}, Kami{\'n}ska and Kubiak computed the dual norm of the Ces\`{a}ro function space $\ces_{p}(u)$, generated by $1 < p < \infty$ and an arbitrary positive weight $u$. A description presented in \cite{kamkub} resembles the approach of Jagers \cite{jagers} for sequence spaces.  

Let $X$ and $Y$ be a (quasi)-Banach spaces of Borel measurable functions on $\I$. A {\it multiplier} from a space $X$ into a space $Y$ is a function $f$ such that $f \cdot g \in Y$ for all $g \in X$. In particular, the K\"{o}the dual $X'$ of $X$ is defined as the space $\m (X, L_1)$ of multipliers into $L_1$. The (linear) space of all such multipliers is denoted by $\m(X,Y)$, that is,
$$
\m (X,Y) : = \{f:~ fg \in Y~\mbox{for all}~g \in X\}.
$$
The $\m (X,Y)$ becomes a quasi-normed space with the quantity
$$
\|f\|_{\m(X,Y)} : = \sup_{g \not\sim 0} \frac{\|fg\|_{Y}}{\|g\|_X}.
$$
Given $f \in \W(0,\infty)$, if we define a weighted space by $Y_f = \{g:\, f \cdot g \in Y\}$, then
$$
\|f\|_{\m (X,Y)} = \sup_{g \not \sim  0} \frac{\|g\|_{Y_f}}{\|g\|_X} = \|I\|_{X \rw Y_f},
$$
and we see that $\|f\|_{\m (X,Y)} < \infty$ if and only if $\|I\|_{X \rw Y_f} < \infty$.

The multiplier problem between $\ell_s(w)$ and one of the $\Ces_{p,q}(a,b)$  and $\Cop_{p,q}(a,b)$ spaces was considered by Grosse-Erdmann \cite{grosse}. It is mentioned in \cite[p. 30]{grosse} that multipliers between two spaces of type $\Ces_{p,q}(a,b)$  and $\Cop_{p,q}(a,b)$ are more difficult to treat.

In \cite{gmu} the embeddings between the Copson and Ces\`{a}ro function spaces and vice versa, that is, the embeddings
\begin{align}
\cop_{p_1,q_1}(u_1,v_1) & \hra \ces_{p_2,q_2}(u_2,v_2), \label{mainemb1}\\
\ces_{p_2,q_2}(u_2,v_2) & \hra \cop_{p_1,q_1}(u_1,v_1). \label{mainemb2}
\end{align}
were characterized. Using these characterizations we give the solution  to the multiplier problem between weighted Ces\`{a}ro and Copson function spaces.

It is easy to see that the problem of characterization of the space of multipliers between $\cop_{p_1,q_1}(u_1,v_1)$ and $\ces_{p_2,q_2}(u_2,v_2)$ can be  
reduced to the characterization of the multiplier problem with three parameters and three weight functions (see, Proposition \ref{prop.reduc}). So we will concentrate our attention on calculation of $M(\cop_{r}(u),\ces_{p,q}(w,v))$.

In order to describe the pointwise multipliers, we need to introduce the following weighted Ces\`{a}ro-type spaces with three parameters:
\begin{defi}\label{defi.2.2}
	Let  $0 <p, \, q, \,r \le \infty$, $u \in \mp^+ \I$, $v,\,w\in \W\I$. The weighted Ces\`{a}ro and Copson type spaces are defined by
	\begin{align*}
	\ces_{p,q,r} (u,v,w) : & = \bigg\{ f \in \mp^+ \I: \|f\|_{\ces_{p,q,r}(u,v,w)} : = \bigg\| \big\| \|f\|_{p,w,(0,\cdot)} \big\|_{q,v,(0,\cdot)} \bigg\|_{r,u,\I} < \i \bigg\}, \\
	\intertext{and} 	\cop_{p,q,r} (u,v,w) : & = \bigg\{ f \in \mp^+ \I: \|f\|_{\cop_{p,q,r}(u,v,w)} : = \bigg\| \big\| \|f\|_{p,w,(\cdot,\infty)} \big\|_{q,v,(\cdot,\infty)} \bigg\|_{r,u,\I} < \infty \bigg\},
	\end{align*}
	respectively.
\end{defi}

Throughout the paper we assume that $I : = (a,b)\subseteq
(0,\infty)$. By ${\mathfrak M} (I)$ we denote the set of all
measurable functions on $I$. The symbol ${\mathfrak M}^+ (I)$ stands
for the collection of all $f\in{\mathfrak M} (I)$ which are
non-negative on $I$, while ${\mathfrak M}^+ (I;\downarrow)$ and
${\mathfrak M}^+ (I;\uparrow)$ are used to denote the subset of
those functions which are non-increasing and non-decreasing on $I$,
respectively. When $I = (0,\infty)$, we write simply ${\mathfrak M}^+$,
${\mathfrak M}^{\downarrow}$ and ${\mathfrak M}^{\uparrow}$ instead of ${\mathfrak M}^+ (I)$,
${\mathfrak M}^+ (I;\downarrow)$ and ${\mathfrak M}^+ (I;\uparrow)$,
accordingly. The family of all weight functions (also called just weights) on $I$, that is, measurable, positive and finite a.e. on $I$, is given by $\W (I)$.  

For $p\in (0,\i]$, we define the functional $\|\cdot\|_{p,I}$ on $\mp (I)$ by 
\begin{equation*}
\|f\|_{p,I} : = \left\{
\begin{array}{cl}
	\bigg(\int_I |f(x)|^p \,dx \bigg)^{\frac{1}{p}} & \qq\mbox{if}\qq p<\i, \\
	\esup_{I} |f(x)| & \qq\mbox{if}\qq p=\i.
\end{array}
\right.
\end{equation*}

If $w\in \W(I)$, then the weighted Lebesgue space $L_p(w,I)$ is given by
\begin{equation*}
L_p(w,I) \equiv L_{p,w}(I) : = \{f\in \mp (I):\,\, \|f\|_{p,w,I} : =
\|fw\|_{p,I} < \i\},
\end{equation*}
and it is equipped with the quasi-norm $\|\cdot\|_{p,w,I}$. When $I = \I$, we often write simply $L_{p,w}$ and $L_p(w)$ instead of $L_{p,w}(I)$ and $L_p(w,I)$, respectively.

We adopt the following usual conventions.
\begin{conv}\label{Notat.and.prelim.conv.1.1}
{\rm (i)} Throughout the paper we put $0/0 = 0$, $0 \cdot (\pm \i) = 0$ and $1 / (\pm\i) =0$.

{\rm (ii)} We put
$$
p' : = \left\{\begin{array}{cl} \frac p{1-p} & \text{if} \quad 0<p<1,\\
\infty &\text{if}\quad p=1, \\
\frac p{p-1}  &\text{if}\quad 1<p<\infty,\\
1  &\text{if}\quad p=\infty.
\end{array}
\right.
$$

{\rm (iii)} If $I = (a,b) \subseteq \R$ and $g$ is a monotone function on $I$, then by $g(a)$ and $g(b)$ we mean the limits $\lim_{x\rw a+}g(x)$ and $\lim_{x\rw b-}g(x)$, respectively.
\end{conv}

To state our results we use the notation $p \rw q$ for $0 < p,\,q \le \infty$ defined by
$$
\frac{1}{p \rw q} = \frac{1}{q} - \frac{1}{p} \qq \mbox{if} \qq q <
p,
$$
and $p \rw q = \infty$ if $q \ge p$ (see, for instance, \cite[p.30]{grosse}).

Throughout the paper, we always denote by $c$ and $C$ a positive constant, which is independent of main parameters but it may vary from line to line. However a constant with subscript or superscript such as $c_1$ does not change in different occurrences. By $a\lesssim b$, ($b\gtrsim a$) we mean that $a\leq \la b$, where $\la>0$ depends on inessential parameters. If $a\lesssim b$ and $b\lesssim a$, we write $a\approx b$ and say that $a$ and $b$ are equivalent.  We will denote by $\bf 1$ the function ${\bf 1}(x) = 1$, $x \in \R$.  Given two quasi-normed vector spaces $X$ and $Y$, we write $X=Y$ if $X$ and $Y$ are equal in the algebraic and the topological sense (their quasi-norms are equivalent). The symbol $X\hookrightarrow Y$ ($Y \hookleftarrow X$) means that $X\subset Y$ and the natural embedding $\Id$ of $X$ in $Y$ is continuous, that is, there exist a constant $c > 0$ such that $\|z\|_Y \le c\|z\|_X$ for all $z\in X$. The best constant of the embedding $X\hookrightarrow Y$ is $\|\Id\|_{X \rw Y}$.

The paper is organized as follows. Some new "gluing" lemmas are presented in Section \ref{gluing}. The characterization of spaces of multipliers between weighted Ces\`{a}ro and Copson function spaces are given in Section \ref{multipliers}.


\section{Some new "gluing" lemmas}\label{gluing}

In this section, we present some generalizations of "gluing" lemma, which is very useful  and has independent interest (cf. \cite[Lemma 2.2]{gkp_2009} and \cite[Theorem 3.1]{gogperstepwall}).

Recall that, if $F \in {\mathfrak M}^+ (\I;\downarrow)$, then
\begin{equation}\label{Fubini.1}
\esup_{t \in (0,\infty)} F(t)G(t) = \esup_{t \in (0,\infty)} F(t)
\esup_{\tau \in (0,t)} G(\tau);	
\end{equation}
likewise, when $F \in {\mathfrak M}^+ (\I;\uparrow)$, then
\begin{equation}\label{Fubini.2}	
\esup_{t \in (0,\infty)} F(t)G(t) = \esup_{t \in (0,\infty)} F(t)
\esup_{\tau \in (t,\infty)} G(\tau)
\end{equation}
(see, for instance, \cite[p. 85]{gp_2006}).

Given $a \in {\mathfrak M}^+ (\I;\uparrow)$ denote by
$$
{\mathcal A}(x,t) : = \frac{a(x)}{a(x) + a(t)} \qquad (x > 0,\, t > 0).
$$
Observe that
$$
{\mathcal A}(x,t) \approx \min \bigg\{ 1,\frac{a(x)}{a(t)} \bigg\}.
$$	
Moreover,
\begin{equation}\label{prop.1}
a(t){\mathcal A}(x,t) = a(x){\mathcal A}(t,x).
\end{equation}

We say that a function $f$ is $a$-quasiconcave if $f$ is equivalent to 
an increasing function on $(0,\infty)$ and $f / a$ is equivalent to a decreasing function on
$(0,\infty)$. It is easy to see that ${\mathcal A}(x,t)$ is $a$-quasiconcave function of $x$ for any fixed $t > 0$.

It have been shown in \cite[p. 85]{gp_2006} that the relations
\begin{align}
\esup_{t \in (0,\infty)}{\mathcal A}(x,t) g(t) & \approx \esup_{t \in (0,\infty)} g(t) \min \bigg\{ 1,\frac{a(x)}{a(t)} \bigg\} \notag \\
& = \esup_{t \in (0,x)} a(t) \esup_{\tau \in (t,\infty)} \frac{g(\tau)}{a(\tau)} \notag \\
& = a(x) \esup_{t \in (x,\infty)} \frac{1}{a(t)} \esup_{\tau \in (0,t)} g(\tau)   \label{Fubini.3}
\end{align}
holds for any $g \in {\mathfrak M}^+ (0,\infty)$ and $a \in {\mathfrak M}^+ (\I;\uparrow)$. Consequently,  $\esup_{t \in (0,\infty)}{\mathcal A}(x,t) g(t)$ is $a$-quasiconcave function. 

\begin{lem}\label{gluing.lem.1}
Let $a \in \W\I$ be non-decreasing and $g,\,h \in \M^+\I$. Then
\begin{align*}
\esup_{x \in (0,\infty)} \bigg( \esup_{t \in (0,\infty)} {\mathcal A}(x,t) g(t) \bigg) \bigg( \esup_{t\in \I} {\mathcal A}(t,x) h(t)\bigg) & \\
& \hspace{-5.5cm} \ap \esup_{x \in (0,\infty)} g(x) \bigg( \esup_{t \in (x,\infty)} h(t) \bigg)  \\
& \hspace{-5cm} + \esup_{x \in (0,\infty)} a(x)^{-1} g(x) \bigg(\esup_{t\in (0,x)} a(t) h(t)\bigg).
\end{align*}
\end{lem}

\begin{proof}
Using the second relation in \eqref{Fubini.3}, by \eqref{Fubini.1}, \eqref{prop.1}, and \eqref{Fubini.2}, we obtain that
\begin{align*}
\esup_{x \in (0,\infty)} \bigg( \esup_{t \in (0,\infty)} {\mathcal A}(x,t)g(t) \bigg) \bigg(\esup_{t\in \I} {\mathcal A}(t,x) h(t)\bigg)  & \\
& \hspace{-5cm } \approx \esup_{x \in (0,\infty)} g(x)  \bigg(\esup_{t\in \I} {\mathcal A}(t,x) h(t)\bigg). 
\end{align*}
It remains to observe that
$$
\esup_{t\in \I} {\mathcal A}(t,x) h(t) \approx a(x)^{-1} \esup_{t \in (0,x)} a(t)h(t) + \esup_{t \in (x,\infty)} h(t).
$$
\end{proof}

Assume that $g \in {\mathfrak M}^+ (0,\infty)$, $a \in {\mathfrak M}^+ (\I;\uparrow)$ and $0 < \beta < \infty$. Recall that the function 
$$
\bigg( \int_0^{\infty}  {\mathcal A}(x,t)^{\beta}g(t)\,dt \bigg)^{\frac{1}{\beta}}
$$
is $a$-quasiconcave function, as well (see, for instance, \cite[p. 318]{gp_2003}).

\begin{lem}\label{gluing.lem.2}
Let $\b$ be a positive number. Suppose that $g,\,h \in \M^+\I$ and $a \in \W\I$ is non-decreasing. Then
\begin{align*}
\esup_{x \in (0,\infty)} \bigg( \esup_{t \in (0,\infty)} {\mathcal A}(x,t)g(t) \bigg) \bigg( \int_0^{\infty}  {\mathcal A}(t,x)^{\beta}h(t)\,dt \bigg)^{\frac{1}{\beta}} & \\
& \hspace{-5.5cm} \ap \esup_{x \in (0,\infty)} g(x) \bigg( \int_x^{\infty} h(t)\,dt\bigg)^{\frac{1}{\b}}
+ \esup_{x \in (0,\infty)} a(x)^{-1}g(x)\bigg( \int_0^x a(t)^{\b} h(t)\,dt\bigg)^{\frac{1}{\b}}.
\end{align*}
\end{lem}

\begin{proof}
Using the second relation in \eqref{Fubini.3}, by \eqref{Fubini.1}, \eqref{prop.1}, and \eqref{Fubini.2}, we obtain that	
\begin{align*}
\esup_{x \in (0,\infty)} \bigg( \esup_{t \in (0,\infty)} {\mathcal A}(x,t)g(t) \bigg) \bigg( \int_0^{\infty}  {\mathcal A}(t,x)^{\beta}h(t)\,dt \bigg)^{\frac{1}{\beta}}  & \\
& \hspace{-5cm } \approx \esup_{x \in (0,\infty)} g(x) \bigg( \int_0^{\infty}  {\mathcal A}(t,x)^{\beta}h(t)\,dt \bigg)^{\frac{1}{\beta}}. 
\end{align*}
To complete the proof it remains to observe that
$$
\bigg( \int_0^{\infty}  {\mathcal A}(t,x)^{\beta}h(t)\,dt \bigg)^{\frac{1}{\beta}}  \approx a(x)^{-1}\bigg( \int_0^x a(t)^{\beta} h(t) \,dt \bigg)^{\frac{1}{\beta}} + \bigg( \int_x^{\infty} h(t) \,dt \bigg)^{\frac{1}{\beta}}. 
$$
\end{proof}

\begin{lem}\label{gluing.lem.3}
Let $\b$ be a positive number. Suppose that $g,\,h \in \M^+\I$ and $a \in \W\I$ is non-decreasing.
Then
\begin{align*}
\esup_{x \in (0,\infty)} \bigg( \int_0^{\infty}  {\mathcal A}(x,t)^{\beta}g(t)\,dt \bigg)^{\frac{1}{\beta}} \bigg( \esup_{t \in (0,\infty)} {\mathcal A}(t,x)h(t) \bigg) & \\
& \hspace{-5.5cm} \ap \esup_{x \in (0,\infty)} h(x)\bigg( \int_0^x g(t)\,dt\bigg)^{\frac{1}{\b}} + \esup_{x \in (0,\infty)} a(x)h(x) \bigg( \int_x^{\infty} a(t)^{-\b} g(t)\,dt\bigg)^{\frac{1}{\b}}.
\end{align*}
\end{lem}

\begin{proof}
Using identity \eqref{prop.1}, the statement easily follows from Lemma \ref{gluing.lem.2}. 
\end{proof}

We quote some known results. Proofs  can be found in \cite[Proposition 2.1]{ghs_1996}.

\begin{defi} \label{D:2.1}
	Let $N,M\in \overline{\Z}$, $N<M$. A positive almost  non-increasing
	sequence $\{\tau _k\} _{k=N}^M$ (that is, there exists $K\geq1$ such that $\tau_{n+1}\leq K\tau_n$) is called {\it almost geometrically decreasing} if there are $\alpha \in (1,\infty )$ and $L\in \N$
	such that
	$$
	{\alpha}\tau _k\leq \tau _{k-L} \quad
	\text{for all} \quad k\in \{ N+L, \dots ,M\}. 
	$$
	A positive almost non-decreasing sequence $\{\sigma _k\} _{k=N}^M$ (that is, there exists $K\geq 1$ such that $\sigma_n \leq K\sigma_{n+1}$) is called {\it almost geometrically increasing} if there are $\alpha \in (1,
	+\infty )$ and $L\in \N$ such that
	$$
	\sigma _k\geq \alpha \sigma _{k-L} \quad \text{for all} \quad k\in
	\{ N+L, \dots ,M\}.$$
\end{defi}

\begin{lem}\label{AGD lemma}
	Let $q \in (0,\infty]$, $N,M\in \overline{\Z}$, $N\leq M$, $\mathcal{Z}=\{N,N+1,...,M-1,M\}$ and let $\{\tau_k\}_{k=N}^M$ be an almost geometrically decreasing sequence. Then
	\begin{align}
	\left \| \left\{ \tau_k \sum_{m=N}^k  a_m \right\}  \right \|_{\ell^q(\mathcal{Z})} & \ap \|\{\tau_k a_k\}\|_{\ell^q(\mathcal{Z})} \\
	\intertext{and}
	\left \| \left\{ \tau_k \sup_{N\leq m\leq k} a_m \right\} \right\|_{\ell^q(\mathcal{Z})} & \ap \| \{\tau_k a_k\}\|_{\ell^q(\mathcal{Z})}
	\end{align}
	for all non-negative sequences $\{a_k\}_{k=N}^M$.
\end{lem}

\begin{lem}\label{AGI lemma}
	Let $q\in (0,+\infty ]$, $ N,M\in \overline{\Z}$, $N\le M$ and let $\{ \sigma_k\} _{k=N}^M$ be an almost geometrically increasing sequence. Then
	\begin{align}\label{E1.3}
	\left\| \left\{ \sigma _k \sum_{m=k}^{M} a_m \right\} \right\| _{\ell^q(\mathcal{Z})} & \approx \| \{\sigma _k a_k\}\| _{\ell^q(\mathcal{Z})} \\
	\intertext{and}
	\left \| \left\{ \sigma _k \sup _{k\leq m\leq M} a_m \right\} \right\| _{\ell^q(\mathcal{Z})} & \approx \| \{\sigma _k a_k \}\| _{\ell^q(\mathcal{Z})}
	\end{align}
	for all non-negative sequences $\{ a_k\} _{k=N}^M$.
\end{lem}

\begin{lem}\label{gluing.lem.0}
Let $\a$ and $\b$ be positive numbers. Suppose that $g,\,h \in \M^+\I$ and $a \in \W\I$ is non-decreasing.
Then
\begin{align*}
\esup_{x \in (0,\infty)} \bigg( \int_0^{\infty} {\mathcal A}(x,t)^{\b} g(t)\,dt \bigg)^{\frac{1}{\b}} \bigg(\int_0^{\infty} {\mathcal A}(t,x)^{\a} h(t)\,dt\bigg)^{\frac{1}{\a}} & \\
& \hspace{-5.5cm} \ap \esup_{x \in (0,\infty)} \bigg( \int_0^x g(t)\,dt\bigg)^{\frac{1}{\b}} \bigg(\int_x^{\infty} h(t)\,dt\bigg)^{\frac{1}{\a}}  \\
& \hspace{-5cm} + \esup_{x \in (0,\infty)} \bigg( \int_x^{\infty} a(t)^{-\b} g(t)\,dt \bigg)^{\frac{1}{\b}}  \bigg( \int_0^x a(t)^{\a}h(t)\,dt\bigg)^{\frac{1}{\a}}.
\end{align*}
\end{lem}

\begin{proof}
Denote by
\begin{align*}
A_5 : & = \esup_{x \in (0,\infty)} \bigg( \int_0^x g(t)\,dt\bigg)^{\frac{1}{\b}} \bigg( \int_x^{\infty} h(t)\,dt\bigg)^{\frac{1}{\a}}, \\
A_6 : & = \esup_{x \in (0,\infty)} \bigg( \int_x^{\infty} a(t)^{-\b} g(t)\,dt \bigg)^{\frac{1}{\b}}  \bigg( \int_0^x a(t)^{\a}h(t)\,dt\bigg)^{\frac{1}{\a}}.
\end{align*}
Obviously,
\begin{align*}
\esup_{x \in (0,\infty)} \bigg( \int_0^{\infty} {\mathcal A}(x,t)^{\b} g(t)\,dt \bigg)^{\frac{1}{\b}} \bigg(\int_0^{\infty} {\mathcal A}(t,x)^{\a} h(t)\,dt\bigg)^{\frac{1}{\a}} & \\
& \hspace{-7cm} \ap \esup_{x \in (0,\infty)} \bigg( \int_0^x g(t)\,dt + a(x)^{\b} \int_x^{\infty} a(t)^{-\b}g(t) \,dt\bigg)^{\frac{1}{\b}} \bigg(a(x)^{-\a} \int_0^x a(t)^{\a}h(t)\,dt + \int_x^{\infty} h(t)\,dt\bigg)^{\frac{1}{\a}} \\
& \hspace{-7cm} \ap A_5 + A_6 + B_5 + B_6,
\end{align*}
where
\begin{align*}
B_5 : & = \esup_{x \in (0,\infty)} \bigg( \int_0^x g(t)\,dt\bigg)^{\frac{1}{\b}} a(x)^{-1} \bigg( \int_0^x a(t)^{\a}h(t)\,dt\bigg)^{\frac{1}{\a}}, \\
B_6 : & = \esup_{x \in (0,\infty)} a(x) \bigg(\int_x^{\infty} a(t)^{-\b}g(t) \,dt\bigg)^{\frac{1}{\b}} \bigg(\int_x^{\infty} h(t)\,dt\bigg)^{\frac{1}{\a}}.
\end{align*}
It is enough to show that $B_i \ls A_5 + A_6$, $i = 5,6$.
	
First, let us show that $B_5 \ls A_5 + A_6$. We will consider the case when $\int_0^{\infty} g(t)\,dt < \infty$ (The case when $\int_0^{\infty} g(t)\,dt = \infty$ is much simpler to treat). Define a sequence $\{x_m\}_{m = -\infty}^M$ such that $\int_0^{x_m} g(t)\,dt = 2^m$ if $-\infty < m \le M$ and $2^M \le \int_0^{\infty} g(t)\,dt < 2^{M+1}$. Denote by $x_{M+1} : = \infty$. Then, by \eqref{Fubini.2} and Lemma~\ref{AGI lemma}, we have that
\begin{align*}
B_5 & = \esup_{x \in (0,\infty)} \bigg( \int_0^x g(t)\,dt\bigg)^{\frac{1}{\b}} \esup_{y \in (x,\infty)} a(y)^{-1} \bigg( \int_0^y a(t)^{\a}h(t)\,dt\bigg)^{\frac{1}{\a}} \\
& \ap \sup_{- \infty < m \le M} 2^{\frac{m}{\b}} \esup_{y \in (x_m,\infty)} a(y)^{-1} \bigg( \int_0^y a(t)^{\a}h(t)\,dt\bigg)^{\frac{1}{\a}} \\
& \ap \sup_{- \infty < m \le M} 2^{\frac{m}{\b}} \esup_{y \in (x_m,x_{m+1})} a(y)^{-1} \bigg( \int_0^y a(t)^{\a}h(t)\,dt\bigg)^{\frac{1}{\a}}.
\end{align*}
For every $-\infty < m \le M$, there exists $y_m \in (x_m,x_{m+1})$ such that
$$
\esup_{y \in (x_m,x_{m+1})} a(y)^{-1} \bigg( \int_0^y a(t)^{\a}h(t)\,dt\bigg)^{\frac{1}{\a}} \le 2 a(y_m)^{-1} \bigg( \int_0^{y_m} a(t)^{\a}h(t)\,dt\bigg)^{\frac{1}{\a}}.
$$
Therefore,
\begin{align*}
B_5 & \ls \sup_{- \infty < m \le M} 2^{\frac{m}{\b}} a(y_m)^{-1} \bigg( \int_0^{y_m} a(t)^{\a}h(t)\,dt\bigg)^{\frac{1}{\a}} \\
& \ap \sup_{- \infty < m \le M} 2^{\frac{m}{\b}} a(y_m)^{-1} \bigg( \int_0^{y_{m-2}} a(t)^{\a}h(t)\,dt\bigg)^{\frac{1}{\a}} + \sup_{- \infty < m \le M} 2^{\frac{m}{\b}} a(y_m)^{-1} \bigg( \int_{y_{m-2}}^{y_m} a(t)^{\a}h(t)\,dt\bigg)^{\frac{1}{\a}} = : I + II.
\end{align*}
Note that $\int_{y_{m-2}}^{y_m} g(t)\,dt \ap 2^m$, $- \infty < m \le M$. It yields that
\begin{align*}
I & \ap  \sup_{- \infty < m \le M} \bigg( \int_{y_{m-2}}^{y_m} g(t)\,dt \bigg)^{\frac{1}{\b}} a(y_m)^{-1} \bigg( \int_0^{y_{m-2}} a(t)^{\a}h(t)\,dt\bigg)^{\frac{1}{\a}} \\
& = \sup_{- \infty < m \le M} \esup_{y_{m-2} < x < y_m} \bigg( \int_x^{y_m} g(t)\,dt \bigg)^{\frac{1}{\b}} a(y_m)^{-1} \bigg( \int_0^{y_{m-2}} a(t)^{\a}h(t)\,dt\bigg)^{\frac{1}{\a}} \\
& \leq \sup_{- \infty < m \le M} \esup_{y_{m-2} < x < y_m} \bigg( \int_x^{y_m} a(t)^{-\b}g(t)\,dt \bigg)^{\frac{1}{\b}}  \bigg( \int_0^{y_{m-2}} a(t)^{\a}h(t)\,dt\bigg)^{\frac{1}{\a}} \\
& \leq \sup_{- \infty < m \le M} \esup_{y_{m-2} < x < y_m} \bigg( \int_x^{\infty} a(t)^{-\b}g(t)\,dt \bigg)^{\frac{1}{\b}}  \bigg( \int_0^x a(t)^{\a}h(t)\,dt\bigg)^{\frac{1}{\a}} \\	
& \lesssim \esup_{x \in (0,\infty)} \bigg( \int_x^{\infty} a(t)^{-\b} g(t)\,dt\bigg)^{\frac{1}{\b}}  \bigg( \int_0^x a(t)^{\a}h(t)\,dt\bigg)^{\frac{1}{\a}} = A_6.
\end{align*}
For $II$ we have that
\begin{align*}
II & \ap  \sup_{- \infty < m \le M} \bigg( \int_{y_{m-4}}^{y_{m-2}} g(t)\,dt \bigg)^{\frac{1}{\b}} a(y_m)^{-1} \bigg( \int_{y_{m-2}}^{y_m} a(t)^{\a}h(t)\,dt\bigg)^{\frac{1}{\a}} \\
& = \sup_{- \infty < m \le M}  \esup_{y_{m-4} < x < y_{m-2}} \bigg( \int_{y_{m-4}}^x g(t)\,dt \bigg)^{\frac{1}{\b}} a(y_m)^{-1} \bigg( \int_{y_{m-2}}^{y_m} a(t)^{\a}h(t)\,dt\bigg)^{\frac{1}{\a}} \\
& \le \sup_{- \infty < m \le M}  \esup_{y_{m-4} < x < y_{m-2}} \bigg( \int_{y_{m-4}}^x g(t)\,dt \bigg)^{\frac{1}{\b}} \bigg( \int_{y_{m-2}}^{y_m} h(t)\,dt\bigg)^{\frac{1}{\a}} \\
& \le \sup_{- \infty < m \le M}  \esup_{y_{m-4} < x < y_{m-2}} \bigg( \int_{y_{m-4}}^x g(t)\,dt \bigg)^{\frac{1}{\b}} \bigg( \int_x^{y_m} h(t)\,dt\bigg)^{\frac{1}{\a}} \\	
& \le \sup_{- \infty < m \le M}  \esup_{y_{m-4} < x < y_{m-2}} \bigg( \int_0^x g(t)\,dt \bigg)^{\frac{1}{\b}} \bigg( \int_x^{\infty} h(t)\,dt\bigg)^{\frac{1}{\a}} \\	
& \le \esup_{x \in (0,\infty)} \bigg( \int_0^x g(t)\,dt\bigg)^{\frac{1}{\b}} \bigg( \int_x^{\infty} h(t)\,dt\bigg)^{\frac{1}{\a}} = A_5.
\end{align*}
Combining, we get that $B_5 \ls A_5 + A_6$.
	
Now we show that $B_6 \ls A_5 + A_6$. Let $\int_0^{\infty} a(t)^{-\b}g(t)\,dt < \infty$ (It is much simpler to deal with the case when $\int_0^{\infty} a(t)^{-\b}g(t)\,dt = \infty$). Define  a sequence $\{x_m\}_{m = N}^{\infty}$ such that $\int_{x_m}^{\infty} a(t)^{-\b} g(t)\,dt = 2^{-m}$ if $N \le m < \infty$ and $2^{-N} < \int_0^{\infty} a(t)^{-\b} g(t)\,dt \le 2^{- N + 1}$. Denote by $x_{N-1} : = 0$.

By using Lemma~\ref{AGD lemma}, we find that
\begin{align*}
B_6 & = \esup_{x \in (0,\infty)} \bigg(\int_x^{\infty} a(t)^{-\b}g(t) \,dt\bigg)^{\frac{1}{\b}} \esup_{y \in (0,x)} a(y)\bigg(\int_y^{\infty} h(t)\,dt\bigg)^{\frac{1}{\a}} \\
& \ap \sup_{N \le m < \infty}  2^{- \frac{m}{\b}} \esup_{y \in (0,x_m)} a(y)\bigg(\int_y^{\infty} h(t)\,dt\bigg)^{\frac{1}{\a}} \\
& \ap \sup_{N \le m < \infty}  2^{- \frac{m}{\b}} \esup_{y \in (x_{m-1},x_m)} a(y)\bigg(\int_y^{\infty} h(t)\,dt\bigg)^{\frac{1}{\a}}.
\end{align*}
For every $N \leq m < \infty$, there exists $y_m \in (x_{m-1},x_m)$ such that
$$
\esup_{y \in (x_{m-1},x_m)} a(y)\bigg(\int_y^{\infty} h(t)\,dt\bigg)^{\frac{1}{\a}} \le 2  a(y_m)\bigg(\int_{y_m}^{\infty} h(t)\,dt\bigg)^{\frac{1}{\a}}.
$$
Hence
\begin{align*}
B_6 & \ls \sup_{N \le m < \infty}  2^{- \frac{m}{\b}} a(y_m)\bigg(\int_{y_m}^{\infty} h(t)\,dt\bigg)^{\frac{1}{\a}} \\
& \ap \sup_{N \le m < \infty}  2^{- \frac{m}{\b}} a(y_m)\bigg(\int_{y_m}^{y_{m+2}} h(t)\,dt\bigg)^{\frac{1}{\a}} +  \sup_{N \le m < \infty}  2^{- \frac{m}{\b}} a(y_m)\bigg(\int_{y_{m+2}}^{\infty} h(t)\,dt\bigg)^{\frac{1}{\a}} = : III + IV.
\end{align*}
Since $\int_{y_m}^{y_{m+2}} a(t)^{-\b} g(t)\,dt \ap 2^{-m}$, we have that
\begin{align*}
III & \ap \sup_{N \le m < \infty}  \bigg( \int_{y_{m+2}}^{y_{m+4}} a(t)^{-\b} g(t)\,dt \bigg)^{\frac{1}{\b}} a(y_m)\bigg(\int_{y_m}^{y_{m+2}} h(t)\,dt\bigg)^{\frac{1}{\a}} \\
& = \sup_{N \le m < \infty}  \esup_{y_{m+2} < x < y_{m+4}}\bigg( \int_x^{y_{m+4}} a(t)^{-\b} g(t)\,dt \bigg)^{\frac{1}{\b}} a(y_m)\bigg(\int_{y_m}^{y_{m+2}} h(t)\,dt\bigg)^{\frac{1}{\a}} \\
& \leq \sup_{N \le m < \infty}  \esup_{y_{m+2} < x < y_{m+4}}\bigg( \int_x^{y_{m+4}} a(t)^{-\b} g(t)\,dt \bigg)^{\frac{1}{\b}} \bigg(\int_{y_m}^{y_{m+2}} a(t)^{\a}h(t)\,dt\bigg)^{\frac{1}{\a}} \\	
& \leq \sup_{N \le m < \infty}  \esup_{y_{m+2} < x < y_{m+4}}\bigg( \int_x^{y_{m+4}} a(t)^{-\b} g(t)\,dt \bigg)^{\frac{1}{\b}} \bigg(\int_{y_m}^x a(t)^{\a}h(t)\,dt\bigg)^{\frac{1}{\a}} \\	
& \leq \sup_{N \le m < \infty}  \esup_{y_{m+2} < x < y_{m+4}}\bigg( \int_x^{\infty} a(t)^{-\b} g(t)\,dt \bigg)^{\frac{1}{\b}} \bigg(\int_0^x a(t)^{\a}h(t)\,dt\bigg)^{\frac{1}{\a}} \\		
& \approx \esup_{x \in (0,\infty)} \bigg( \int_x^{\infty} a(t)^{-\b} g(t)\,dt \bigg)^{\frac{1}{\b}}  \bigg( \int_0^x a(t)^{\a}h(t)\,dt\bigg)^{\frac{1}{\a}} =  A_6.
\end{align*}
Moreover,
\begin{align*}
IV & \ap \sup_{N \le m < \infty}  \bigg( \int_{y_m}^{y_{m+2}} a(t)^{-\b}g(t)\,dt \bigg)^{\frac{1}{\b}} a(y_m)\bigg(\int_{y_{m+2}}^{\infty} h(t)\,dt\bigg)^{\frac{1}{\a}} \\
& \le \sup_{N \le m < \infty}  \bigg( \int_{y_m}^{y_{m+2}} g(t)\,dt \bigg)^{\frac{1}{\b}} \bigg(\int_{y_{m+2}}^{\infty} h(t)\,dt\bigg)^{\frac{1}{\a}} \\
& \le \sup_{N \le m < \infty}  \bigg( \int_0^{y_{m+2}} g(t)\,dt \bigg)^{\frac{1}{\b}} \bigg(\int_{y_{m+2}}^{\infty} h(t)\,dt\bigg)^{\frac{1}{\a}} \\
& \lesssim \esup_{x \in (0,\infty)} \bigg( \int_0^x g(t)\,dt\bigg)^{\frac{1}{\b}} \bigg( \int_x^{\infty} h(t)\,dt\bigg)^{\frac{1}{\a}} = A_5.
\end{align*}
Therefore, we obtain $B_6 \ls A_5 + A_6$. The proof is complete.
\end{proof}

\begin{lem}\label{gluing.lem.4}
Let $\a,\,\b,\,\gamma$ be positive numbers. Suppose that $g,\,h \in \M^+\I$ and $a \in \W\I$ is non-decreasing. Then
\begin{align*}
\int_0^{\infty} \bigg( \int_0^{\infty} {\mathcal A}(x,t)^{\a}g(t)\,dt \bigg)^{\frac{\gamma}{\a} - 1} \bigg( \int_0^{\infty} {\mathcal A}(t,x)^{\b} h(t)\,dt\bigg)^{\frac{\gamma}{\b}}g(x)\,dx & \\
& \hspace{-5.5cm} \ap \int_0^{\infty} \bigg( \int_0^x g(t)\,dt\bigg)^{\frac{\gamma}{\a} - 1} \bigg( \int_x^{\infty} h(t)\,dt\bigg)^{\frac{\gamma}{\b}} g(x)\,dx \\
& \hspace{-5cm} + \int_0^{\infty} \bigg( \int_x^{\infty} a(t)^{-\a} g(t)\,dt \bigg)^{\frac{\gamma}{\a} - 1}  \bigg( \int_0^x  a(t)^{\b}h(t)\,dt\bigg)^{\frac{\gamma}{\b}}\,a(x)^{-\a}g(x)\,dx.
\end{align*}
\end{lem}

\begin{proof}
Denote by
\begin{align*}
A_7 : & = \int_0^{\infty} \bigg( \int_0^x g(t)\,dt\bigg)^{\frac{\gamma}{\a} - 1} \bigg( \int_x^{\infty} h(t)\,dt\bigg)^{\frac{\gamma}{\b}} g(x)\,dx, \\
A_8 : & = \int_0^{\infty} \bigg( \int_x^{\infty} a(t)^{-\a} g(t)\,dt \bigg)^{\frac{\gamma}{\a} - 1} \bigg( \int_0^x  a(t)^{\b}h(t)\,dt\bigg)^{\frac{\gamma}{\b}}\,a(x)^{-\a}g(x)\,dx.
\end{align*}
Obviously,
\begin{align*}
\int_0^{\infty} \bigg( \int_0^{\infty} {\mathcal A}(x,t)^{\a}g(t)\,dt \bigg)^{\frac{\gamma}{\a} - 1} \bigg( \int_0^{\infty} {\mathcal A}(t,x)^{\b} h(t)\,dt\bigg)^{\frac{\gamma}{\b}}g(x)\,dx & \\
& \hspace{-9.5cm} \ap \int_0^{\infty} \bigg( \int_0^x g(t)\,dt + a(x)^{\a} \int_x^{\infty} a(t)^{-\a}g(t) \,dt\bigg)^{\frac{\gamma}{\a} - 1} \bigg(a(x)^{-\b} \int_0^x a(t)^{\b}h(t)\,dt + \int_x^{\infty} h(t)\,dt\bigg)^{\frac{\gamma}{\b}} g(x)\,dx \\
& \hspace{-9.5cm} \ap A_7 + A_8 + B_7 + B_8,
\end{align*}
where
\begin{align*}
B_7 : & = \int_0^{\infty} \bigg( \int_0^x g(t)\,dt\bigg)^{\frac{\gamma}{\a} - 1} a(x)^{-\gamma} \bigg( \int_0^x a(t)^{\b}h(t)\,dt\bigg)^{\frac{\gamma}{\b}} g(x)\,dx, \\
B_8 : & = \int_0^{\infty} \bigg( \int_x^{\infty} a(t)^{-\a}g(t) \,dt\bigg)^{\frac{\gamma}{\a} - 1} \bigg(\int_x^{\infty} h(t)\,dt\bigg)^{\frac{\gamma}{\b}} a(x)^{\gamma - \a} g(x)\,dx.
\end{align*}
It is enough to show that $B_i \ls A_7 + A_8$, $i = 7,8$.
	
First, let us show that $B_7 \ls A_7 + A_8$. Let $\int_0^{\infty} g(t)\,dt < \infty$ (The case when $\int_0^{\infty} g(t)\,dt = \infty$ is much simpler to treat). Define a sequence $\{x_m\}_{m = -\infty}^M$ such that $\int_0^{x_m} g(t)\,dt = 2^m$ if $-\infty < m \le M$ and $2^M \le \int_0^{\infty} g(t)\,dt < 2^{M+1}$. Denote by $x_{M+1} : = \infty$. Then we have by Lemma~\ref{AGI lemma} that
\begin{align*}
B_7 & \le \int_0^{\infty} \bigg( \int_0^x g(t)\,dt\bigg)^{\frac{\gamma}{\a} - 1} \bigg( \esup_{y \in (x,\infty) }a(y)^{-\gamma} \bigg( \int_0^y a(t)^{\b}h(t)\,dt\bigg)^{\frac{\gamma}{\b}} \bigg) \, g(x)\,dx \\
& \ap \sum_{m = -\infty}^M 2^{m\frac{\gamma}{\a}} \esup_{y \in (x_m,\infty) }a(y)^{-\gamma} \bigg( \int_0^y a(t)^{\b}h(t)\,dt\bigg)^{\frac{\gamma}{\b}} \\
& \ap \sum_{m = -\infty}^M 2^{m\frac{\gamma}{\a}} \esup_{y \in (x_m,x_{m+1}) }a(y)^{-\gamma} \bigg( \int_0^y a(t)^{\b}h(t)\,dt\bigg)^{\frac{\gamma}{\b}}.
\end{align*}
For every $-\infty < m \le M$ there exists $y_m \in (x_m,x_{m+1})$ such that
$$
\esup_{y \in (x_m,x_{m+1})} a(y)^{-\gamma} \bigg( \int_0^y a(t)^{\b}h(t)\,dt\bigg)^{\frac{\gamma}{\b}} \le 2 a(y_m)^{-\gamma} \bigg( \int_0^{y_m} a(t)^{\b}h(t)\,dt\bigg)^{\frac{\gamma}{\b}}.
$$
Therefore,
\begin{align*}
B_7 & \ls \sum_{m = -\infty}^M 2^{m\frac{\gamma}{\a}} a(y_m)^{-\gamma} \bigg( \int_0^{y_m} a(t)^{\b}h(t)\,dt\bigg)^{\frac{\gamma}{\b}} \\
& \ap \sum_{m = -\infty}^M 2^{m\frac{\gamma}{\a}} a(y_m)^{-\gamma} \bigg( \int_0^{y_{m-2}} a(t)^{\b}h(t)\,dt\bigg)^{\frac{\gamma}{\b}} + \sum_{m = -\infty}^M 2^{m\frac{\gamma}{\a}} a(y_m)^{-\gamma} \bigg( \int_{y_{m-2}}^{y_m} a(t)^{\b}h(t)\,dt\bigg)^{\frac{\gamma}{\b}} = : V + VI.
\end{align*}
Since $\int_{y_{m-2}}^{y_m} \bigg( \int_x^{y_m} g(t)\,dt \bigg)^{\frac{\gamma}{\a}-1} g(x)\,dx \ap 2^{m \frac{\gamma}{\a}}$ for $- \infty < m \le M$, we have that
\begin{align*}
V & \ap  \sum_{m = -\infty}^M \bigg( \int_{y_{m-2}}^{y_m} \bigg( \int_x^{y_m} g(t)\,dt \bigg)^{\frac{\gamma}{\a}-1} g(x)\,dx \bigg) \,\, a(y_m)^{-\gamma} \bigg( \int_0^{y_{m-2}} a(t)^{\b}h(t)\,dt\bigg)^{\frac{\gamma}{\b}} \\
& \le \sum_{m = -\infty}^M \bigg( \int_{y_{m-2}}^{y_m} \bigg( \int_x^{y_m} a(t)^{-\a}g(t)\,dt \bigg)^{\frac{\gamma}{\a}-1} a(x)^{-a} g(x)\,dx \bigg) \,\, \bigg( \int_0^{y_{m-2}} a(t)^{\b}h(t)\,dt\bigg)^{\frac{\gamma}{\b}} \\
& \le \sum_{m = -\infty}^M \int_{y_{m-2}}^{y_m} \bigg( \int_x^{\infty} a(t)^{-\a}g(t)\,dt \bigg)^{\frac{\gamma}{\a}-1} \bigg( \int_0^x a(t)^{\b}h(t)\,dt\bigg)^{\frac{\gamma}{\b}} a(x)^{-a} g(x)\,dx \\
& \ls \int_0^{\infty} \bigg( \int_x^{\infty} a(\tau)^{-\a} g(\tau)\,d\tau\bigg)^{\frac{\gamma}{\a} - 1}  \bigg( \int_0^x  a(t)^{\b}h(t)\,dt\bigg)^{\frac{\gamma}{\b}}\,a(x)^{-\a}g(x)\,dx = A_8.
\end{align*}
In view of $\int_{y_{m-4}}^{y_{m-2}} \bigg( \int_{y_{m-4}}^x g(t)\,dt \bigg)^{\frac{\gamma}{\a}-1} g(x)\,dx \ap 2^{m \frac{\gamma}{\a}}$, $- \infty < m \le M$, we get that
\begin{align*}
VI & \ap  \sum_{m = -\infty}^M \bigg( \int_{y_{m-4}}^{y_{m-2}} \bigg( \int_{y_{m-4}}^x g(t)\,dt \bigg)^{\frac{\gamma}{\a}-1} g(x)\,dx \bigg) \,\, a(y_m)^{-\gamma} \bigg( \int_{y_{m-2}}^{y_m} a(t)^{\b}h(t)\,dt\bigg)^{\frac{\gamma}{\b}} \\
& \le \sum_{m = -\infty}^M \bigg( \int_{y_{m-4}}^{y_{m-2}} \bigg( \int_{y_{m-4}}^x g(t)\,dt \bigg)^{\frac{\gamma}{\a}-1} g(x)\,dx \bigg) \, \bigg( \int_{y_{m-2}}^{y_m} h(t)\,dt\bigg)^{\frac{\gamma}{\b}} \\
& \le \sum_{m = -\infty}^M \int_{y_{m-4}}^{y_{m-2}} \bigg( \int_0^x g(t)\,dt \bigg)^{\frac{\gamma}{\a}-1} \bigg( \int_x^{\infty} h(t)\,dt\bigg)^{\frac{\gamma}{\b}} g(x)\,dx \\
& \ls \int_0^{\infty} \bigg( \int_0^x g(t)\,dt\bigg)^{\frac{\gamma}{\a} - 1} \bigg( \int_x^{\infty} h(t)\,dt\bigg)^{\frac{\gamma}{\b}} g(x)\,dx = A_7.
\end{align*}
Combining, we get that $B_7 \ls A_7 + A_8$.

Now we show that $B_8 \ls A_7 + A_8$. Let $\int_0^{\infty} a(t)^{-\a}g(t)\,dt < \infty$ (It is much simpler to deal with the case when $\int_0^{\infty} a(t)^{-\a}g(t)\,dt = \infty$). Define  a sequence $\{x_m\}_{m = N}^{\infty}$ such that $\int_{x_m}^{\infty} a(t)^{-\a}g(t)\,dt = 2^{-m}$ if $N \le m < \infty$ and $2^{-N} < \int_0^{\infty} a(t)^{-\a}g(t)\,dt \le 2^{- N + 1}$. Denote by $x_{N-1} : = 0$.
By using elementary calculations, in view of Lemma~\ref{AGD lemma}, we find that
\begin{align*}
B_8 & \le \int_0^{\infty} \bigg( \int_x^{\infty} a(t)^{-\a}g(t) \,dt\bigg)^{\frac{\gamma}{\a} - 1} \bigg( \esup_{y \in (0,x)} a(y)^{\gamma} \bigg(\int_y^{\infty} h(t)\,dt\bigg)^{\frac{\gamma}{\b}} \bigg) \,a(x)^{-\a} g(x)\,dx \\
& \ap \sum_{m = N}^{\infty} 2^{-m\frac{\gamma}{\a}} \esup_{y \in (0,x_m)} a(y)^{\gamma} \bigg(\int_y^{\infty} h(t)\,dt\bigg)^{\frac{\gamma}{\b}} \\
& \ap \sum_{m = N}^{\infty} 2^{-m\frac{\gamma}{\a}} \esup_{y \in (x_{m-1},x_m)} a(y)^{\gamma} \bigg(\int_y^{\infty} h(t)\,dt\bigg)^{\frac{\gamma}{\b}}. 
\end{align*}
For every $m = N,\, N+1,\ldots$ there exists $y_m \in (x_{m-1},x_m)$ such that
$$
\esup_{y \in (x_{m-1},x_m)} a(y)^{\gamma}\bigg(\int_y^{\infty} h(t)\,dt\bigg)^{\frac{\gamma}{\b}} \le 2  a(y_m)^{\gamma}\bigg(\int_{y_m}^{\infty} h(t)\,dt\bigg)^{\frac{\gamma}{\b}}.
$$
Hence
\begin{align*}
B_8 & \ls \sum_{m = N}^{\infty} 2^{-m\frac{\gamma}{\a}} a(y_m)^{\gamma}\bigg(\int_{y_m}^{\infty} h(t)\,dt\bigg)^{\frac{\gamma}{\b}} \\
& \ap \sum_{m = N}^{\infty} 2^{-m\frac{\gamma}{\a}} a(y_m)^{\gamma}\bigg(\int_{y_m}^{y_{m+2}} h(t)\,dt\bigg)^{\frac{\gamma}{\b}} + \sum_{m = N}^{\infty} 2^{-m\frac{\gamma}{\a}} a(y_m)^{\gamma}\bigg(\int_{y_{m+2}}^{\infty} h(t)\,dt\bigg)^{\frac{\gamma}{\b}} = : VII + VIII.
\end{align*}
Since $\int_{y_{m+2}}^{y_{m+4}} \bigg( \int_x^{y_{m+4}} a(t)^{-\a} g(t)\,dt \bigg)^{\frac{\gamma}{\a}-1} a(x)^{-\a} g(x)\,dx \ap 2^{-m \frac{\gamma}{\a}}$, $N \le m < \infty$, we have that
\begin{align*}
VII & \ap \sum_{m = N}^{\infty} \bigg( \int_{y_{m+2}}^{y_{m+4}} \bigg( \int_x^{y_{m+4}} a(t)^{-\a} g(t)\,dt \bigg)^{\frac{\gamma}{\a}-1} a(x)^{-\a} g(x)\,dx \bigg) \,\, a(y_m)^{\gamma}\bigg(\int_{y_m}^{y_{m+2}} h(t)\,dt\bigg)^{\frac{\gamma}{\b}} \\
& \le \sum_{m = N}^{\infty} \bigg( \int_{y_{m+2}}^{y_{m+4}} \bigg( \int_x^{y_{m+4}} a(t)^{-\a} g(t)\,dt \bigg)^{\frac{\gamma}{\a}-1} a(x)^{-\a} g(x)\,dx \bigg) \, \bigg(\int_{y_m}^{y_{m+2}} a(t)^{\b}h(t)\,dt\bigg)^{\frac{\gamma}{\b}} \\
& \le \sum_{m = N}^{\infty} \int_{y_{m+2}}^{y_{m+4}} \bigg( \int_x^{\infty} a(t)^{-\a} g(t)\,dt \bigg)^{\frac{\gamma}{\a}-1} \bigg(\int_0^x a(t)^{\b}h(t)\,dt\bigg)^{\frac{\gamma}{\b}} \,a(x)^{-\a} g(x)\,dx \\
& \ls \int_0^{\infty} \bigg( \int_x^{\infty} a(\tau)^{-\a} g(\tau)\,d\tau\bigg)^{\frac{\gamma}{\a} - 1}  \bigg( \int_0^x  a(t)^{\b}h(t)\,dt\bigg)^{\frac{\gamma}{\b}}\,a(x)^{-\a}g(x)\,dx = A_8.
\end{align*}
In view of $\int_{y_m}^{y_{m+2}} \bigg( \int_{y_m}^x a(t)^{-\a} g(t)\,dt \bigg)^{\frac{\gamma}{\a}-1} a(x)^{-\a} g(x)\,dx \ap 2^{-m \frac{\gamma}{\a}}$, $N \le m < \infty$, we get that
\begin{align*}
VIII & \ap \sum_{m = N}^{\infty} \bigg( \int_{y_m}^{y_{m+2}} \bigg( \int_{y_m}^x a(t)^{-\a} g(t)\,dt \bigg)^{\frac{\gamma}{\a}-1} a(x)^{-\a} g(x)\,dx \bigg) \,\, a(y_m)^{\gamma}\bigg(\int_{y_{m+2}}^{\infty} h(t)\,dt\bigg)^{\frac{\gamma}{\b}} \\
& \le \sum_{m = N}^{\infty} \bigg( \int_{y_m}^{y_{m+2}} \bigg( \int_{y_m}^x g(t)\,dt \bigg)^{\frac{\gamma}{\a}-1} g(x)\,dx \bigg) \, \bigg(\int_{y_{m+2}}^{\infty} h(t)\,dt\bigg)^{\frac{\gamma}{\b}} \\
& \le \sum_{m = N}^{\infty} \int_{y_m}^{y_{m+2}} \bigg( \int_0^x g(t)\,dt \bigg)^{\frac{\gamma}{\a}-1} \bigg(\int_x^{\infty} h(t)\,dt\bigg)^{\frac{\gamma}{\b}} \, g(x)\,dx \\
& \ls \int_0^{\infty} \bigg( \int_0^x g(t)\,dt\bigg)^{\frac{\gamma}{\a} - 1} \bigg( \int_x^{\infty} h(t)\,dt\bigg)^{\frac{\gamma}{\b}} g(x)\,dx = A_7.
\end{align*}
Therefore, we obtain $B_8 \ls A_7 + A_8$. The proof is complete.	
\end{proof}

Finally, we recall the following statement from \cite{gmu}.
\begin{lem}\cite[Lemma 3.13]{gmu}\label{gluing.lem}
Let $\b$ be a positive number. Suppose that $g,\,h \in \M^+\I$ and $a \in \W\I$ is non-decreasing.
Then
\begin{align*}
\int_0^{\infty} \bigg( \int_0^{\infty} {\mathcal A}(x,t)g(t)\,dt \bigg)^{\b - 1} \bigg(\esup_{t\in \I} {\mathcal A}(t,x) h(t)\bigg)^{\b}g(x)\,dx & \\
& \hspace{-5.5cm} \ap \int_0^{\infty} \bigg( \int_0^x g(t)\,dt\bigg)^{\b - 1} \bigg(\esup_{t \in (x, \infty)} h(t)\bigg)^{\b} g(x)\,dx \\
& \hspace{-5cm} + \int_0^{\infty} \bigg( \int_x^{\infty} a(t)^{-1} g(t)\,dt\bigg)^{\b - 1}  \bigg( \esup_{t \in (0,x)} a(t)h(t)\bigg)^{\b}\,a(x)^{-1}g(x)\,dx.
\end{align*}
\end{lem}

\section{Multipliers}\label{multipliers}

To state further results we need the following definitions.
\begin{defi}
	Let $U$ be a continuous, strictly increasing function on $[0,\i)$ such that $U(0)=0$ and $\lim_{t \rw \infty} U(t) = \infty$. Then we say that $U$ is admissible.
\end{defi}

Let $U$ be an admissible function. We say that a $U$-quasiconcave function $\vp$ is non-degenerate if
$$
\lim_{t \rw 0+} \vp(t) = \lim_{t \rw \infty} \frac{1}{\vp(t)} = \lim_{t \rw \infty} \frac{\vp(t)}{U(t)} = \lim_{t \rw 0+} \frac{U(t)}{\vp(t)} = 0.
$$
The family of non-degenerate $U$-quasiconcave functions is denoted by $Q_U$.
\begin{defi}
	Let $U$ be an admissible function, and let $w$ be a non-negative measurable function on $(0,\i)$. We say that the function $\vp$, defined by 
	\begin{equation*}
	\vp(t)=U(t)\int_0^{\infty} \frac{w(\tau)\,d\tau}{U(\tau)+U(t)}, \qq t\in (0,\infty),
	\end{equation*}
	is a fundamental function of $w$ with respect to $U$. One will also say that $w(\tau)\,d\tau$ is a representation measure of $\vp$ with respect to $U$.
\end{defi}

\begin{rem}\label{nondegrem}
	Let $\vp$ be the fundamental function of $w$ with respect to $U$. Assume that
	\begin{equation*}
	\int_0^{\infty}\frac{w(\tau)\,d\tau}{U(\tau)+U(t)} < \infty, ~ t> 0, \qq \int_0^1 \frac{w(\tau)\,d\tau}{U(\tau)}=\int_1^{\infty}w(\tau)\,d\tau=\infty.
	\end{equation*}
	Then $\vp\in Q_{U}$.
\end{rem}

\begin{rem}\label{limsupcondition}
	Suppose that $\vp (x) < \infty$ for all $x \in (0,\infty)$, where $\vp$ is defined by
	\begin{equation*}
	\vp(x)=\esup_{t\in(0,x)}{U(t)} \esup_{\tau\in(t,\infty)}\frac{w(\tau)}{U(\tau)},~~t\in\I.
	\end{equation*}
	If
	$$
	\limsup_{t \rightarrow 0 +} w(t) = \limsup_{t \rightarrow +\infty} \frac{1}{w(t)} = \limsup_{t \rightarrow 0 +} \frac{U(t)}{w(t)} = \limsup_{t \rightarrow +\infty} \frac{w(t)}{U(t)} = 0,
	$$
	then
	$\vp\in Q_{U}$.
\end{rem}

We are going to use the following operators
$$
A_{q,p}: \mp^+ \I \rw \mp^+ \I \quad \mbox{and} \quad A_{q,p}^*: \mp^+ \I \rw \mp^+ \I 
$$
defined by
\begin{align*}
A_{q,p}(u)(x): = \bigg( \int_0^x u^q \bigg)^{-\frac{1}{p}} u^{\frac{q-p}{p}}(x) \quad \mbox{and} \quad A_{q,p}^*(u)(x): = \bigg( \int_x^{\infty} u^q \bigg)^{\frac{1}{p}} u^{\frac{p-q}{p}}(x)
\end{align*}
for all $x > 0$.

In order to state our results we will use the following notations:
$$
 V(x) : = \| v \|_{p' ,(0,x)} \quad \mbox{and} \quad {\mathcal V}(x,t):= \frac{V(x)}{V(x) + V(t)} \qquad (t > 0,\,x > 0).
$$

The following simple stataement allows us to reduce the characterization of the multiplier problem between weighted Copson function spaces  $\cop_{p_1,q_1}(u_1,v_1)$ and weighted Ces\`{a}ro function spaces $\ces_{p_2,q_2}(u_2,v_2)$ to the problem with three parameters and three weight functions.
\begin{prop}\label{prop.reduc}
	Let $0 <p_1,\,p_2,\,q_1,\,q_2 < \infty$ and $v_1,\,v_2 \in \W\I$, $u_1 \in \dual{\O_{q_1}}$ and $u_2 \in \O_{q_2}$. Then	
	$$
	\|f\|_{M\big(\cop_{p_1,q_1}(u_1,v_1),\ces_{p_2,q_2}(u_2,v_2)\big)}  = \big\|f^{p_1}\big\|_{M\bigg( \cop_{\frac{q_1}{p_1}}\big(u_1^{p_1}\big), \ces_{\frac{p_2}{p_1},\frac{q_2}{p_1}}\big(u_2^{p_1}, v_1^{-p_1}v_2^{p_1}\big) \bigg)}^{\frac{1}{p_1}}.
	$$	
\end{prop}
\begin{proof}
	Indeed:
	\begin{align*}
	\|f\|_{M\big(\cop_{p_1,q_1}(u_1,v_1),\ces_{p_2,q_2}(u_2,v_2)\big)} & = \sup_{g \not\sim 0} \frac{\|fg\|_{\ces_{p_2,q_2}(u_2,v_2)}} {\|g\|_{\cop_{p_1,q_1}(u_1,v_1)}} \\
	&= \sup_{g \not\sim 0} \frac{\bigg(\int_0^{\infty} \bigg(\int_0^t f^{p_2} g^{p_2} v_2^{p_2} \bigg)^{\frac{q_2}{p_2}} u_2^{q_2}(t) dt \bigg)^{\frac{1}{q_2}}} {\bigg(\int_0^{\infty} \bigg(\int_t^{\infty} g^{p_1}  v_1^{p_1} \bigg)^{\frac{q_1}{p_1}} u_1^{q_1}(t) dt \bigg)^{\frac{1}{q_1}}} \\
	& = \sup_{h \not\sim 0} \frac{\bigg(\int_0^{\infty} \bigg(\int_0^t f^{p_2} h^{\frac{p_2}{p_1}} v_1^{-p_2} v_2^{p_2} \bigg)^{\frac{q_2}{p_2}} u_2^{q_2}(t) dt \bigg)^{ \frac{1}{q_2} } } { \bigg( \int_0^{\infty} \bigg( \int_t^{\infty} h \bigg)^{ \frac{q_1}{p_1} } u_1^{q_1}(t) dt \bigg)^{ \frac{1}{q_1} }}\\
	& = \sup_{h \not\sim 0} \frac{\bigg(\int_0^{\infty} \bigg(\int_0^t \big[f^{p_1} h\big]^{\frac{p_2}{p_1}} \big[ v_1^{-p_1} v_2^{p_1} \big]^\frac{p_2}{p_1} \bigg)^{ \frac{p_1}{p_2} \frac{q_2}{p_1}} \big[u_2(t)^{p_1}\big]^\frac{q_2}{p_1} dt \bigg)^{ \frac{p_1}{q_2} \frac{1}{p_1} } } { \bigg( \int_0^{\infty} \bigg( \int_t^{\infty} h \bigg)^{ \frac{q_1}{p_1} } \big[u_1(t)^{p_1}\big]^{\frac{q_1}{p_1}} dt \bigg)^{\frac{p_1}{q_1} \frac{1}{p_1} }} \\
	& = \big\|f^{p_1}\big\|_{M\bigg( \cop_{\frac{q_1}{p_1}}\big(u_1^{p_1}\big), \ces_{\frac{p_2}{p_1},\frac{q_2}{p_1}}\big(u_2^{p_1}, v_1^{-p_1}v_2^{p_1}\big) \bigg)}^{\frac{1}{p_1}}.
	\end{align*} 
\end{proof}

Now we present and prove our main results. We start with simple cases.
\begin{thm}
Let $0< q =  p \le r=1 $, $v \in \W\I$, $u \in \dual{\O_{1}}$ and $w \in \O_{q}$. Then
\begin{equation*}
M(\cop_{r}(u),\ces_{p,q}(w,v)) = L_{p'}(\omega),
\end{equation*}
where
$$
\omega : = A_{1,1}(u) \cdot A_{p,p}^*(w)\cdot v.
$$

Moreover
\begin{align*}
\|f\|_{M(\cop_{r}(u),\ces_{p,q}(w,v))} \ap \, \|f\|_{p',\omega,(0,\infty)}.
\end{align*}
\end{thm}

\begin{proof}
By \cite[Theorem 2.1]{gmu}, we have that
\begin{align*}
\|f\|_{M(\cop_{r}(u),\ces_{p,q}(w,v))} & = \|\Id\|_{\cop_{r}(u)\rightarrow \ces_{p,q}(w,f v)} \\
& \ap \big\| \|u\|_{1,(0,\cdot)}^{- 1} \|w\|_{p,(\cdot,\infty)} \big\|_{p', f v,\I} \\
& = \|f\|_{p', \omega, (0,\infty)}.
\end{align*}
\end{proof}

\begin{thm}
Let $0 < p, q, r < \infty$, $q \neq p$ and $p \leq r = 1$. Let  $v  \in \W\I$, $u \in \dual{\O_{r}}$ and $w \in \O_{q}$.

\begin{itemize}
\item[(i)] If $1\le q$, then
\begin{equation*}
M(\cop_{r}(u),\ces_{p,q}(w,v)) =  \ces_{p',\infty}(\omega_1,\omega_2),
\end{equation*}
where
$$
\omega_1 : = A_{p,p}^*(w), \quad \omega_2 : = A_{1,1}(u) \cdot v.
$$

Moreover
\begin{align*}
\|f\|_{M(\cop_{r}(u),\ces_{p,q}(w,v))} \ap \, \|f\|_{\ces_{p',\infty}(\omega_1,\omega_2)}.
\end{align*}

\item[(ii)] If $q < 1$, then
\begin{align*}
M(\cop_{r}(u),\ces_{p,q}(w,v)) =  \ces_{p', q'} (\omega_1, \omega_2),
\end{align*}
where
$$
\omega_1 : = A_{q,1}^* (w), \quad \omega_2 : = A_{1,1}(u) \cdot v.
$$

Moreover
\begin{align*}
\|f\|_{M(\cop_{r}(u),\ces_{p,q}(w,v))} \ap \, \|f\|_{\ces_{p', q'} (\omega_1, \omega_2)}.
\end{align*}
\end{itemize}
\end{thm}

\begin{proof}
{\rm (i)} Let $1 \le q$. By \cite[Theorem 2.2, (i)]{gmu}, we get that
\begin{align*}
\|f\|_{M(\cop_{r}(u),\ces_{p,q}(w,v))}  & = \|\Id\|_{\cop_{r}(u)\rightarrow \ces_{p,q}(w,f v)} \\
& \ap \sup_{t \in \I} \big\| \|u\|_{1,(0,\cdot)}^{-1} \big\|_{p',fv,(0,t)} \|w\|_{q,(t,\infty)} \\
& \ap \sup_{t \in \I} \big\| f \big\|_{p',\|u\|_{1,(0,\cdot)}^{-1} v, (0,t)} \|w\|_{q,(t,\infty)} \\
& = \big\| \| f \|_{p',\omega_2,(0,\cdot)} \big\|_{\infty,\omega_1,(0,\infty)} \\
& = \|f\|_{\ces_{p',\infty}(\omega_1,\omega_2)}.
\end{align*}

{\rm (ii)}  Let $q < 1$. By \cite[Theorem 2.2, (ii)]{gmu}, we obtain that
\begin{align*}
\|f\|_{M(\cop_{r}(u),\ces_{p,q}(w,v))}  & = \|\Id\|_{\cop_{r}(u)\rightarrow \ces_{p,q}(w,f v)} \\
& \ap \bigg(\int_{(0,\infty)} \big\| \|u\|_{1,(0,\cdot)}^{-1}	\big\|_{p',fv,(0,t)}^{q'} d \,\bigg(- \|w\|_{q,(t,\infty)}^{q'} \bigg)\bigg)^{\frac{1}{q'}} \\
& = \bigg(\int_{(0,\infty)} \| f \|_{p', \|u\|_{1,(0,\cdot)}^{-1}\,v,(0,t)}^{q'} d \,\bigg(- \|w\|_{q,(t,\infty)}^{q'} \bigg)\bigg)^{\frac{1}{q'}} \\
& = \big\| \| f \|_{p',\omega_2,(0,\cdot)} \big\|_{q',\omega_1,(0,\infty)} \\
& = \|f\|_{\ces_{ p', q'} (\omega_1,\omega_2)}.
\end{align*}
\end{proof}

\begin{thm}
Let $0 < p,q,r < \infty$, $r \neq 1$ and $p = q \leq 1$. Let $v \in \W\I$, $u \in \dual{\O_{r}}$ and $w \in \O_{q}$.
\begin{itemize}
\item[(i)] If $r \le p$, then
\begin{equation*}
M(\cop_{r}(u),\ces_{p,q}(w,v)) =  \ces_{p',\infty}(\omega_1, \omega_2),
\end{equation*}
where
$$
\omega_1 : = A_{r,r}(u), \quad \omega_2 : = A_{p,p}^* (w) \cdot v.
$$
Moreover,
\begin{align*}
\|f\|_{M(\cop_{r}(u),\ces_{p,q}(w,v))} \ap \, \|f\|_{\ces_{p',\infty}(\omega_1, \omega_2)}.
\end{align*}

\item[(ii)] If $ p < r$, then
\begin{align*}
M(\cop_{r}(u),\ces_{p,q}(w,v))  & \\
&\hspace{-4cm}= \left\{
\begin{array}{cc}
\ces_{p', r \rw p}(\omega_1, \omega_2) \bigcap L_{p'}(\omega_2) & ~ \mbox{if}  \qq \|u\|_{r,(0,\infty)} < \infty, \\
\ces_{p', r \rw p}(\omega_1, \omega_2)   & ~ \mbox{if}  \qq \|u\|_{r,(0,\infty)} =\infty,
\end{array}
\right.
\end{align*}
where
$$
\omega_1 : = A_{r,p}(u), \quad  \omega_2 : = A_{p,p}^* (w) v.
$$
Moreover,
$$
\|f\|_{M(\cop_{r}(u),\ces_{p,q}(w,v))} \ap \|f\|_{\ces_{p', r \rw p}(\omega_1, \omega_2)} +
\|u\|_{r,(0,\infty)} \|f\|_{p',\omega_2,(0,\infty)}.
$$
\end{itemize}
\end{thm}

\begin{proof}
{\rm (i)} Let $r \le p$. By \cite[Theorem 2.3, (i)]{gmu}, we get that
\begin{align*}
\|f\|_{M(\cop_{r}(u),\ces_{p,q}(w,v))} & \ap \sup_{t \in \I}  \|u\|_{r,(0,t)}^{-1} \,\big\| \|w\|_{p,(\cdot,\infty)} \big\|_{1 \rw p,f v,(0,t)} \\
& = \sup_{t \in \I}  \|u\|_{r,(0,t)}^{-1} \,\big\| f \big\|_{p',\|w\|_{p,(\cdot,\infty)}\,v,(0,t)} \\
& = \big\| \|f\|_{p',\omega_2,(0,\cdot)} \big\|_{\infty,\omega_1,(0,\infty)} \\
& = \|f\|_{\ces_{p',\infty}(\omega_1,\omega_2)}.
\end{align*}

{\rm (ii)} Let $p < r$. By \cite[Theorem 2.3, (ii)]{gmu}, we get that
\begin{align*}
\|f\|_{M(\cop_{r}(u),\ces_{p,q}(w,v))} \ap & \, \bigg(\int_{(0,\infty)} \big\| \|w\|_{p,(\cdot,\infty)}\big\|_{1 \rw p, f \, v,(0,t)}^{r \rw p}  d \,\bigg( -\|u\|_{r,(0,t)}^{- r \rw p}\bigg)\bigg)^{\frac{1}{r \rw p}} \\
& + \|u\|_{r,(0,\infty)}^{-1} \,\big\| \|w\|_{p,(\cdot,\infty)}\big\|_{1 \rw p, f \, v,(0,\infty)} \\
= & \, \bigg(\int_{(0,\infty)}\| f \|_{p', \|w\|_{p,(\cdot,\infty)}\, v,(0,t)}^{r \rw p}  d \,\bigg( -\|u\|_{r,(0,t)}^{- r \rw p}\bigg)\bigg)^{\frac{1}{r \rw p}} \\
& + \|u\|_{r,(0,\infty)}^{-1} \,\| f \|_{p', \|w\|_{p,(\cdot,\infty)} \, v,(0,\infty)} \\
= & \, \|f\|_{\ces_{p',r\rw p_2}(\omega_1,\omega_2)} + \|u\|_{r,(0,\infty)}^{-1} \,\| f \|_{p',\omega_2,(0,\infty)}. 
\end{align*}
\end{proof}

\begin{thm}\label{maintheorem1.1}
Let $0 <p,q,r < \infty$, $p < 1$ and $ r \leq p < q $. Let $v \in \W\I$, $u \in \dual{\O_{r}}$ and $w \in \O_{q}$. Assume that $V$ is admissible and $\vp_1 \in Q_{V^{{1} / p'}}$,
where
$$
\vp_1(x):= \esup_{t\in \I} {\mathcal V}(x,t) V(t) \|u\|_{r,(0,t)}^{-1}.
$$
\begin{itemize}
\item[(i)] If $ 1 \leq q$, then
\begin{equation*}
M(\cop_{r}(u),\ces_{p,q}(w,v)) = \ces_{p',\infty}(\omega,v),
\end{equation*}
where
$$
\omega : = A_{r,r}(u) \cdot A_{q,q}^*(w).
$$
Moreover,
\begin{align*}
\|f\|_{M(\cop_{r}(u),\ces_{p,q}(w,v))} \ap \|f\|_{\ces_{p',\infty}(\omega,v)}.
\end{align*}

\item[(ii)] If $ q < 1$, then
\begin{equation*}
M(\cop_{r}(u),\ces_{p,q}(w,v)) = \ces_{p',\infty}(\omega_1,v) \cap
\ces_{p',q',\infty}(\omega_2,\omega_3,v),
\end{equation*}
where
\begin{align*}
\omega_1 : = A_{r,r}(u) \cdot A_{q,q}^*(w),  \quad
\omega_2 : = A_{r,r}(u), \quad \omega_3 : = A_{q,1}^* (w).
\end{align*}

Moreover
\begin{align*}
\|f\|_{M(\cop_{r}(u),\ces_{p,q}(w,v))} \ap \|f\|_{\ces_{p',\infty}(\omega_1,v)} + 
\|f\|_{\ces_{p',q',\infty}(\omega_2,\omega_3,v)}.
\end{align*}
\end{itemize}
\end{thm}

\begin{proof}
{\rm (i)} By \cite[Theorem 2.8, (i)]{gmu}, we have that
\begin{align*}
\|I\|_{\cop_{r}(u) \rw \ces_{p,q}(w,v)}\ap \sup_{x \in (0,\infty)} \bigg(\esup_{t\in \I} {\mathcal V}(x,t) V(t) \|u\|_{r,(0,t)}^{-1}\bigg) \bigg( \sup_{t\in \I} {\mathcal V}(t,x) \|w\|_{q,(t,\infty)} \bigg)
\end{align*}

Using Lemma~\ref{gluing.lem.1} with
$$
a(t) = V(t), \quad g(t) =  V(t)  \|u\|_{r,(0,t)}^{-1}, \quad \mbox{and} \quad h(t) = \|w\|_{q,(t,\infty)} \quad (t > 0),
$$
in view of \eqref{Fubini.1}, we obtain that
\begin{align*}
\|\Id\|_{\cop_{r}(u) \rw \ces_{p,q}(w,v)}  \ap  \esup_{x\in \I} V(x) \|u\|_{r,(0,x)}^{-1} \|w\|_{q,(x,\infty)}.
\end{align*}

Therefore,
\begin{align*}
\|f\|_{M(\cop_{r}(u), \ces_{p,q}(w,v))} & \ap \esup_{x\in \I} \|v f\|_{p',(0,x)} \|u\|_{r,(0,x)}^{-1} \|w\|_{q,(x,\infty)} \\
& = \esup_{x\in \I} \| f\|_{p',v,(0,x)} \|u\|_{r,(0,x)}^{-1} \|w\|_{q,(x,\infty)}\\
&=\|f\|_{\ces_{p',\infty}(\omega,v)}.
\end{align*}

{\rm (ii)} By \cite[Theorem 2.8, (ii)]{gmu}, we have that
\begin{align*}
\| \Id \|_{\cop_{r}(u) \rw \ces_{p,q}(w,v)} &\\
&\hspace{-4cm} \ap \esup_{x\in \I} \bigg( \esup_{t\in \I} {\mathcal V}(x,t) V(t) \|u\|_{r,(0,t)}^{-1} \bigg) \bigg( \int_0^\infty {\mathcal V}(t,x)^{q'}  \bigg( \int_t^\i w^q \bigg)^{q'} w^q(t) dt \bigg)^{\frac{1}{q'}}.
\end{align*}
Using Lemma \ref{gluing.lem.2} with
$$
\b = q', ~ a(t) = V(t), ~ g(t) = V(t)  \|u\|_{r,(0,t)}^{-1}, ~ \mbox{and} ~ h(t) = \bigg( \int_t^\i w^{q} \bigg)^{q'} w^q(t) dt \quad (t > 0),
$$
in view of \eqref{Fubini.1}, we arrive at
\begin{align*}
\|\Id\|_{\cop_{r}(u) \rw \ces_{p,q}(w,v)}  & \\
& \hspace{-3cm} \ap \esup_{x\in \I} V(x) \|u\|_{r,(0,x)}^{-1} \bigg( \int_x^\i \bigg( \int_t^\i w^{q} \bigg)^{q'} w^{q}(t) dt \bigg)^{\frac{1}{q'}} \\
& \hspace{-2.5cm} + \esup_{x\in \I} \|u\|_{r,(0,x)}^{-1} \bigg( \int_0^x V(t)^{q'} \bigg( \int_t^\i w^{q} \bigg)^{q'} w^{q}(t) dt\bigg)^{\frac{1}{q'}}  \\
& \hspace{-3cm} \ap \esup_{x\in \I} V(x) \|u\|_{r,(0,x)}^{-1} \|w\|_{q,(x,\infty)} \\
& \hspace{-2.5cm}  + \esup_{x\in \I} \|u\|_{r,(0,x)}^{-1} \bigg( \int_0^x V(t)^{q'} \bigg( \int_t^\i w^q \bigg)^{q'} w^q(t) dt\bigg)^{\frac{1}{q'}}.
\end{align*}
Hence 
\begin{align*}
\|f\|_{M(\cop_{r}(u), \ces_{p,q}(w,v))} & \\
& \hspace{-3cm} \ap \esup_{x\in \I} \|v f\|_{p',(0,x)} \|u\|_{r,(0,x)}^{-1} \|w\|_{q,(x,\infty)}\\
& \hspace{-2.5cm} + \esup_{x\in \I} \|u\|_{r,(0,x)}^{-1} \bigg( \int_0^x \|v f\|_{p',(0,t)}^{q'} \bigg( \int_t^\i w^q \bigg)^{q'} w^q(t) dt\bigg)^{\frac{1}{q'}} 
\\
& \hspace{-3cm} = \esup_{x\in \I} \|f\|_{p',v,(0,x)} \|u\|_{r,(0,x)}^{-1} \|w\|_{q,(x,\infty)} \\
& \hspace{-2.5cm} + \esup_{x\in \I} \|u\|_{r,(0,x)}^{-1} \bigg( \int_0^x \|f\|_{p', v, (0,t)}^{q'} \bigg( \int_t^\i w^q \bigg)^{q'} w^q(t) dt\bigg)^{\frac{1}{q'}} 
\\
& \hspace{-3cm} = \bigg\| \|f\|_{p',v,(0,x)} \bigg\|_{\infty,\omega_1,(0,\infty)} + 
\bigg\| \big\| \|f\|_{p',v, (0,t)} \big\|_{q', \omega_3,(0,x)} \bigg\|_{\infty,\omega_2,(0,\infty)} \\
& \hspace{-3cm} =  \|f\|_{\ces_{p',\infty}(\omega_1, v)} + 
\|f\|_{\ces_{p',q',\infty}(\omega_2,\omega_3,v)}.
\end{align*}
\end{proof}

\begin{thm}\label{maintheorem1.2}
Let $0 < p,q,r < \infty$, $ p < 1$ and $p < \min\{r,q\}$. Let $v \in \W\I$, $u \in \dual{\O_{r}}$ and $w \in \O_{q}$. Assume that $V$ is admissible and $ \vp_2 \in Q_{V^{{1} / {p'}}}$, where
$$
\vp_2(x):= \bigg(\int_{(0,\infty)} [{\mathcal V}(x,t) V(t)]^{r \rw p} d\bigg( - \|u\|_{r,(0,t)}^{- r \rw p} \bigg) \bigg)^{\frac{1}{r \rw p}}.
$$

{\rm (i)} If $\max\{1,r\} \leq q$, then
\begin{align*}
M(\cop_{r}(u),\ces_{p,q}(w,v)) & \\
& \hspace{-3cm} = \left\{
\begin{array}{cc}
\ces_{p', r \rw p, \infty}(\omega_1, \omega_2,v) \bigcap \ces_{p', \infty}(\omega_3,v) \bigcap \ces_{p',\infty}(\omega_1, v) & ~ \mbox{if} ~ \|u\|_{r,(0,\infty)} < \infty, \\
\ces_{p', r \rw p, \infty}(\omega_1, \omega_2,v) \bigcap \ces_{p', \infty}(\omega_3,v) & ~ \mbox{if} ~  \|u\|_{r,(0,\infty)} =\infty.
\end{array}
\right.
\end{align*} 
where
$$
\omega_1 : = A_{q,q}^* (w), \quad \omega_2 : = A_{r,p}(u), \quad \omega_3 : = A_{r \rw p,r \rw p}^* (A_{r,p}(u)) \cdot A_{q,q}^* (w).
$$
Moreover, 
\begin{align*}
\|f\|_{M(\cop_{r}(u),\ces_{p,q}(w,v))} \ap \|f\|_{\ces_{p', r \rw p, \infty}(\omega_1, \omega_2,v) } + \|f\|_{\ces_{p', \infty}(\omega_3,v)}+ \|u\|_{r,(0,\infty)}^{-1} \|f\|_{\ces_{p',\infty}(\omega_1, v)}.
\end{align*}
	
{\rm (ii)} If $1\le q < r$, then
\begin{align*}
M(\cop_{r}(u),\ces_{p,q}(w,v)) & \\
& \hspace{-3cm} = \left\{
\begin{array}{cc}
\ces_{p',r \rw p, r \rw q}(\omega_1,\omega_2,v) \bigcap \ces_{p',\infty, r \rw q}(\omega_3,\omega_4,v) \bigcap \ces_{p',\infty}(\omega_4, v) & ~ \mbox{if}  ~ \|u\|_{r,(0,\infty)} < \infty, \\
\ces_{p',r \rw p, r \rw q}(\omega_1,\omega_2,v) \bigcap \ces_{p',\infty, r \rw q}(\omega_3,\omega_4,v)   & ~ \mbox{if}  ~ \|u\|_{r,(0,\infty)} =\infty,
\end{array}
\right.
\end{align*}
where
$$
\omega_1 : = A_{q,r}^* (w), \quad \omega_2 : = A_{r,p}(u), \quad \omega_3 : = A_{q,q}^*(w), \quad \omega_4:=\bigg( A_{r \rw p,r \rw p}^* (A_{r,p}(u))\bigg)^{\frac{r \rw p}{q \rw p}} \cdot A_{r,p}(u)^{\frac{r \rw p}{r \rw q}}.
$$
Moreover
\begin{align*}
\|f\|_{M(\cop_{r}(u),\ces_{p,q}(w,v))} & \\
& \hspace{-2cm}\ap  \|f\|_{\ces_{p',r \rw p, r \rw q}(\omega_1,\omega_2,v)} + \|f\|_{\ces_{p',\infty, r \rw q}(\omega_3,\omega_4,v)}  + \|u\|_{r,\I}^{-1} \|f\|_{\ces_{p',\infty}(\omega_4, v)}.
\end{align*}

{\rm (iii)} If $r \le q < 1$, then
\begin{align*}
M(\cop_{r}(u),\ces_{p,q}(w,v)) & \\
& \hspace{-3.5cm} = \left\{
\begin{array}{cc}
\ces_{p',r \rw p, \infty}(\omega_1,\omega_2,v) \bigcap \ces_{p',q', \infty}(\omega_3,\omega_4,v) \bigcap \ces_{p',q'}(\omega_4,v)) & ~ \mbox{if}  ~ \|u\|_{r,(0,\infty)} < \infty, \\
\ces_{p',r \rw p, \infty}(\omega_1,\omega_2,v) \bigcap \ces_{p',q', \infty}(\omega_3,\omega_4,v)   & ~ \mbox{if}  ~ \|u_1\|_{q_1,(0,\infty)} = \infty,
\end{array}
\right.
\end{align*}
where
$$
\omega_1 : = A_{q,q}^* (w), \quad \omega_2 : = A_{r,p}(u), \quad \omega_3 : A_{r \rw p,r \rw p}^* (A_{r,p}(u)), \quad \omega_4:= A_{q,1}^*(w).
$$
Moreover
\begin{align*}
\|f\|_{M(\cop_{r}(u),\ces_{p,q}(w,v))} \ap \|f\|_{\ces_{p',r \rw p, \infty}(\omega_1,\omega_2,v)} + \|f\|_{\ces_{p',q', \infty}(\omega_3,\omega_4,v)}  +  \|u\|_{r,\I}^{-1} \|f\|_{\ces_{p',q'}(\omega_4,v)}.
\end{align*}

{\rm (iv)} If $q < \min\{1,r\}$, then
\begin{align*}
M(\cop_{r}(u),\ces_{p,q}(w,v)) & \\
& \hspace{-4cm} = \left\{
\begin{array}{cc}
\ces_{p', r \rw p, r \rw q}(\omega_1, \omega_2, v) \bigcap \ces_{p', q', r\rw q}(\omega_3, \omega_4, v) \bigcap \ces_{p', q'}(\omega_4, v)) & ~ \mbox{if}  ~ \|u\|_{r,(0,\infty)} < \infty, \\
\ces_{p', r \rw p, r \rw q}(\omega_1, \omega_2, v) \bigcap \ces_{p', q', r\rw q}(\omega_3, \omega_4, v)  & ~ \mbox{if}  ~ \|u\|_{r,(0,\infty)} =\infty,
\end{array}
\right.
\end{align*}
where
$$
\omega_1 : = A_{q,r}^* (w), \quad \omega_2 : = A_{r,p}(u), \quad \omega_3 : =  \bigg( A_{r \rw p,r \rw p}^* (A_{r,p}(u))\bigg)^{\frac{r \rw p}{q \rw p}} \cdot A_{r,p}(u)^{\frac{r \rw p}{r \rw q}}, \quad \omega_4 := A_{q,1}^*(w).
$$
    
Moreover
\begin{align*}
\|f\|_{M(\cop_{r}(u),\ces_{p,q}(w,v))} \ap \|f\|_{\ces_{p', r \rw p, r \rw q}(\omega_1, \omega_2, v)} + \|f\|_{\ces_{p', q', r\rw q}(\omega_3, \omega_4, v)}  +  \|u\|_{r,\I}^{-1} \|f\|_{\ces_{p', q'}(\omega_4, v)}.
\end{align*}    	
\end{thm}

\begin{proof}
{\rm (i)} Let $\max\{1,r\} \leq q$. Applying Lemma \ref{gluing.lem.3} with
$$
\b = r \rw p,~ a(t) =  V(t), ~ g(t) = V(t)^{r \rw p} \bigg( \int_0^t u^r\bigg)^{\frac{-r}{r-p}} u^r(t) ~ \mbox{and} ~ h (t) = \|w\|_{q,(t,\infty)} \quad (t > 0),
$$
by \cite[Theorem 2.9, (i)]{gmu}, we obtain that
\begin{align*}
\|\Id\|_{\cop_r(u) \rw \ces_{p,q}(w,v)} &  \ap \sup_{x \in (0,\infty)} \bigg( \int_{(0,x)} V(t)^{r \rw p} d \bigg( - \|u\|_{r,(0,t)}^{- r \rw	p}\bigg) \bigg)^{\frac{1}{r \rw p}} \|w\|_{q,(x,\infty)}\\
& \hspace{0.5cm} + \sup_{x \in (0,\infty)} \bigg( \int_{(x,\infty)} d \bigg( - \|u\|_{r,(0,t)}^{- r \rw p}\bigg) \bigg)^{\frac{1}{r \rw p}} V(x) \|w\|_{q,(x,\infty)} \\
&  \hspace{1cm} + \|u\|_{r,\I}^{-1} \sup_{t\in \I} V(t)	\|w\|_{q,(t,\infty)} \\
\end{align*}
Consequently, 
\begin{align*}
\|f\|_{M(\cop_{r}(u),\ces_{p,q}(w,v))} & \ap \bigg\| \big\| \|f\|_{p',v,(0,\cdot)} \big\|_{r \rw p,\omega_2,(0,\cdot)} \bigg\|_{\infty, \omega_1,(0,\infty)} + \big\| \|f\|_{p',v,(0,\cdot)} \big\|_{\infty,\omega_3,(0,\infty)}\\
& \hspace{0.5cm}+ \|u\|_{r,(0,\infty)}^{-1} \big\| \|f\|_{p',v,(0,\cdot)} \big\|_{\infty, \omega_1,(0,\infty)}\\
& = \|f\|_{\ces_{p', r \rw p, \infty}(\omega_1,\omega_2,v)} + \|f\|_{\ces_{p',\infty}(\omega_3,v)} \\
& \hspace{0.3cm}+ \|u\|_{r,\I}^{-1} \|f\|_{\ces_{p',\infty}(\omega_1, v)}.
\end{align*}

{\rm (ii)} Let $1\le q < r$. Applying Lemma \ref{gluing.lem} with
$$
\b = \frac{r \rw q}{r \rw p}, ~ a(x) = V(x)^{r \rw p}, ~ g(t) = V(t)^{r \rw p} \bigg( \int_0^t u^r\bigg)^{\frac{-r}{r-p}} u^r(t), ~ \mbox{and} ~ h(t) = \|w\|_{q,(t,\infty)}^{r \rw p},
$$
noting that $\b - 1 = \frac{r \rw q}{q \rw p}$, by \cite[Theorem 2.9, (ii)]{gmu}, we obtain that
\begin{align*}
\|\Id\|_{\cop_r(u) \rw \ces_{p,q}(w,v)} & \\
& \hspace{-2cm} \ap \bigg( \int_{(0,\infty)} \bigg( \int_{(0,x)} V(t)^{r \rw p} d\bigg( - \|u\|_{r,(0,t)}^{- r \rw p} \bigg) \bigg)^{\frac{r \rw q}{q \rw p}} \|w\|_{q,(x,\infty)}^{r \rw q} V(x)^{r \rw p} d\bigg( - \|u\|_{r,(0,x)}^{- r \rw p} \bigg) \bigg)^{\frac{1}{r \rw q}} \\
& \hspace{-1.5cm} + \bigg( \int_{(0,\infty)} \bigg( \sup_{t \in (0,x)} V(t)^{r \rw p} \|w\|_{q,(t,\infty)}^{r \rw p} \bigg)^{\frac{r \rw q}{r \rw p}} \bigg( \int_{(x,\infty)}  d\bigg( - \|u\|_{r,(0,t)}^{- r \rw p} \bigg) \bigg)^{\frac{r \rw q}{q \rw p}} \,d\bigg( - \|u\|_{r,(0,x)}^{- r \rw p} \bigg) \bigg)^{\frac{1}{r \rw q}} \\
&  \hspace{-1.5cm} + \|u\|_{r,\I}^{-1} \sup_{t\in \I} V(t) \|w\|_{q,(t,\infty)}. 
\end{align*}
Integrating by parts in the first integral on the right hand side, we arrive at
\begin{align*}
\|\Id\|_{\cop_r(u) \rw \ces_{p,q}(w,v)} & \\
& \hspace{-2cm} \ap \bigg( \int_{(0,\infty)} \bigg( \int_{(0,x)} V(t)^{r \rw p} d \bigg( - \|u\|_{r,(0,t)}^{- r \rw p} \bigg) \bigg)^{\frac{r \rw q}{r \rw p}} d \bigg(- \|w\|_{q,(x,\infty)}^{r \rw q} \bigg) \bigg)^{\frac{1}{r \rw q}} \\
& \hspace{-1.5cm} + \bigg( \int_{(0,\infty)} \bigg( \sup_{t \in (0,x)} V(t)^{r \rw p} \|w\|_{q,(t,\infty)}^{r \rw p} \bigg)^{\frac{r \rw q}{r \rw p}} \bigg( \int_{(x,\infty)}  d\bigg( - \|u\|_{r,(0,t)}^{- r \rw p} \bigg) \bigg)^{\frac{r \rw q}{q \rw p}} \,d\bigg( - \|u\|_{r,(0,x)}^{- r \rw p} \bigg) \bigg)^{\frac{1}{r \rw q}} \\
&  \hspace{-1.5cm} + \|u\|_{r,\I}^{-1} \sup_{t\in \I} V(t) \|w\|_{q,(t,\infty)}.
\end{align*}   
Thus
\begin{align*}
\|f\|_{M(\cop_{r}(u),\ces_{p,q}(w,v))} \ap & \, \bigg\| \big\| \|f\|_{p',v,(0,\cdot)}\big\|_{r \rw p, \omega_2,(0,\cdot)} \bigg\|_{r \rw q, \omega_1, (0,\infty)}  + \bigg\| \big\| \|f\|_{p',v,(0,\cdot)} \big\|_{\infty,\omega_4,(0,\cdot)} \bigg\|_{r \rw q,\omega_3, (0,\infty)} \\
& + \|u\|_{r,\I}^{-1}  \bigg\| \|f\|_{p',v,(0,\cdot)} \bigg\|_{\infty,\omega_4,(0,\infty)} \\
= & \, \|f\|_{\ces_{p',r \rw p, r \rw q}(\omega_1,\omega_2,v)} + \|f\|_{\ces_{p',\infty, r \rw q}(\omega_3,\omega_4,v)} + \|u\|_{r,\I}^{-1} \|f\|_{\ces_{p',\infty}(\omega_4, v)}.
\end{align*}
     
{\rm (iii)} Let  $r \le q < 1$. Applying Lemma \ref{gluing.lem.0} with
$$
\a = q',~ \b = r \rw p, ~ a(x) = V(x), ~ g(t) = V(t)^{r \rw p} \bigg( \int_0^t u^r\bigg)^{\frac{-r}{r-p}} u^r(t), ~ \mbox{and} ~ h(t) = \bigg( \int_t^\i w^q \bigg)^{q'} w^q(t),
$$
by \cite[Theorem 2.9, (iii)]{gmu}, we obtain that
\begin{align*}
\|\Id\|_{\cop_r(u) \rw \ces_{p,q}(w,v)} \ap & \sup_{x \in (0,\infty)} \bigg( \int_{(0,x)} V(t)^{r \rw p} d \bigg( - \|u\|_{r,(0,t)}^{- r \rw p}\bigg)\bigg)^{\frac{1}{r \rw p}} \bigg(\int_{(x,\infty)} d \bigg( - \|w\|_{q, (t,\infty)}^{r \rw q}\bigg)\bigg)^{\frac{1}{q'}}  \\
& + \sup_{x \in (0,\infty)} \bigg( \int_{(x,\infty)} d\bigg( - \|u\|_{r,(0,t)}^{- r \rw p}\bigg) \bigg)^{\frac{1}{r \rw p}}  \bigg( \int_{(0,x)} V(t)^{q'} d\bigg( - \|w\|_{q, (t,\infty)}^{q'}\bigg) \bigg)^{\frac{1}{q'}} \\
 & \hspace{0.5cm} + \|u\|_{r,\I}^{-1} \bigg( \int_{(0,\infty)} V(t)^{q'} d\bigg( - \|w\|_{q, (t,\infty)}^{q'}\bigg) \bigg)^{\frac{1}{q'}}.
\end{align*}
Consequently
\begin{align*}
\|f\|_{M(\cop_{r}(u),\ces_{p,q}(w,v))}  \ap &  \bigg\| \big\| \|f\|_{p', v, (0,\cdot)} \big\|_{r \rw p, \omega_2, (0,\cdot)} \bigg\|_{\infty,\omega_1,(0,\infty)}  + \bigg\| \big\| \|f\|_{p', v, (0,\cdot)} \big\|_{q', \omega_4,(0,\cdot)} \bigg\|_{\infty,\omega_3,(0,\infty)} \\
& + \|u\|_{r,\I}^{-1} \big\| \|f\|_{p',v,(0,\cdot)} \big\|_{q',\omega_4,(0,\infty)} \\
= & \|f\|_{\ces_{p',r \rw p, \infty}(\omega_1,\omega_2,v)} + \|f\|_{\ces_{p',q', \infty}(\omega_3,\omega_4,v)}  +  \|u\|_{r,\I}^{-1} \|f\|_{\ces_{p',q'}(\omega_4,v)}.
\end{align*}

{\rm (iv)} Let $q < \min\{1,r\}$. Applying Lemma \ref{gluing.lem.4} with 
$$
\a = r \rw p, ~ \b = q',  ~ \ga = r \rw q,  
$$
and
$$
a(x) = V(x), ~ g(t) = V(t)^{r \rw p} \bigg( \int_0^t u^r\bigg)^{\frac{-r}{r-p}} u^r(t), ~h(t) = \bigg( \int_t^\i w^q \bigg)^{q'} w^q(t)
$$
by \cite[Theorem 2.9, (iv)]{gmu}, we obtain that
\begin{align*}
\|\Id\|_{\cop_r(u) \rw \ces_{p,q}(w,v)} & \\
& \hspace{-2.5cm} \ap \bigg(\int_{(0,\infty)} \bigg( \int_{(0,x)} V(t)^{r \rw p} d\bigg( - \|u\|_{r,(0,t)}^{- r \rw p}\bigg)\bigg)^{\frac{r \rw q}{r \rw p} - 1}  \|w\|_{q, (x,\infty)}^{r \rw q} V(x)^{r \rw p} d\bigg( - \|u\|_{r,(0,x)}^{-r \rw	p}\bigg)\bigg)^{\frac{1}{r \rw q}}\\
& \hspace{-2.cm} + \bigg(  \int_{(0,\infty)} \bigg( \int_{(x,\infty)} d \bigg( - \|u\|_{r,(0,t)}^{- r \rw p}\bigg)\bigg)^{\frac{r \rw q}{r \rw p} - 1}  \bigg( \int_{(0,x)}  V(t)^{q'} d \bigg( - \|w\|_{q, (t,\infty)}^{q'}\bigg)\bigg)^{\frac{r \rw q}{q'}}\,d \bigg( - \|u\|_{r,(0,x)}^{-r \rw p}\bigg) \bigg)^{\frac{1}{r \rw q}}\\
& \hspace{-2cm} +\|u\|_{r,\I}^{-1} \bigg( \int_{(0,\infty)} V(t)^{q'} d \bigg( - \|w\|_{q, (t,\infty)}^{q'}\bigg) \bigg)^{\frac{1}{q'}}.
\end{align*}
Integrating by parts at the first integral on the right hand side, we get that
\begin{align*}
\|\Id\|_{\cop_r(u) \rw \ces_{p,q}(w,v)} & \\
& \hspace{-2.5cm} \ap \bigg(\int_{(0,\infty)} \bigg( \int_{(0,x)} V(t)^{r \rw p} d\bigg( - \|u\|_{r,(0,t)}^{- r \rw p}\bigg) \bigg)^{\frac{r \rw q}{r \rw p}} d \bigg( - \|w\|_{q, (x,\infty)}^{r \rw q} \bigg) \bigg)^{\frac{1}{r \rw q}}\\
& \hspace{-2cm} + \bigg(  \int_{(0,\infty)} \bigg( \int_{(x,\infty)} d \bigg( - \|u\|_{r,(0,t)}^{- r \rw p}\bigg)\bigg)^{\frac{r \rw q}{q \rw p}}  \bigg( \int_{(0,x)}  V(t)^{q'} d \bigg( - \|w\|_{q, (t,\infty)}^{q'}\bigg)\bigg)^{\frac{r \rw q}{q'}}\,d \bigg( - \|u\|_{r,(0,x)}^{-r \rw p}\bigg) \bigg)^{\frac{1}{r \rw q}}\\
& \hspace{-2cm} +\|u\|_{r,\I}^{-1} \bigg( \int_{(0,\infty)} V(t)^{q'} d \bigg( - \|w\|_{q, (t,\infty)}^{q'}\bigg) \bigg)^{\frac{1}{q'}}.
\end{align*}
Consequently
\begin{align*}
\|f\|_{M(\cop_{r}(u),\ces_{p,q}(w,v))} \ap & \bigg\| \big\| \|f\|_{p', v, (0,\cdot)} \big\|_{r \rw p, \omega_2, (0,\cdot)} \bigg\|_{r \rw q, \omega_1,(0,\infty)}  + \bigg\| \big\| \|f\|_{p', v, (0,\cdot)} \big\|_{q', \omega_4, (0,\cdot)} \bigg\|_{\infty,\omega_3,(0,\infty)} \\
& + \|u\|_{r,\I}^{-1} \big\| \|f\|_{p',v,(0,\cdot)} \big\|_{q', \omega_4, (0,\infty)} \\
= & \|f\|_{\ces_{p', r \rw p, r \rw q}(\omega_1, \omega_2, v)} + \|f\|_{\ces_{p', q', r\rw q}(\omega_3, \omega_4, v)}  +  \|u\|_{r,\I}^{-1} \|f\|_{\ces_{p', q'}(\omega_4, v)}.
\end{align*}
\end{proof}

\begin{thm}\label{maintheorem1.3}
Let $0 < r < p = 1 < q < \infty$. Assume that $v \in \W\I$, $u \in \dual{\O_{r}}$ and $w \in \O_{q}$. Then
$$
M \big (\cop_{r}(u), \ces_{p,q}(w,v) \big) = L_{\infty}(\omega),
$$
where
$$
\omega : = A_{r,r}(u) \cdot A_{q,q}^* (w) \cdot v.
$$
Moreover
\begin{align*}
\|f\|_{M(\cop_{r}(u), \ces_{p,q}(w,v))} & \ap \|f\|_{\infty,\omega,(0,\infty)}.
\end{align*}
\end{thm}

\begin{proof}
By \cite[Theorem 2.10]{gmu}, on using \eqref{Fubini.1}, we obtain that
\begin{align*}
\|f\|_{M(\cop_{r}(u), \ces_{p,q}(w,v))} & \ap \sup_{t \in \I} \big\| \|u\|_{r,(0,\cdot)}^{-1}\big\|_{\i,f v,(0,t)} \|w\|_{q,(t,\i)} \\
& = \sup_{t \in \I}  f(t) v(t) \|u\|_{r,(0,t)}^{-1} \|w\|_{q,(t,\i)} \\
& = \|f\|_{\infty,\omega,(0,\infty)}.
\end{align*}
\end{proof}

\begin{thm}\label{maintheorem1.4}
Let $0<r, q < \i$, and $1= p < \min\{r,q\}$. Let $u \in \dual{\O_{r}}$ and $w\in \O_{q}$. Assume that $v \in \W\I	\cap C\I$ and $0 < \|w^{-1}\|_{q',(x,\infty)} < \infty,\, x > 0$.
	
{\rm (i)} If $r\le q$, then
\begin{align*}
M(\cop_{r}(u), \ces_{p,q}(w,v)) & \\
& \hspace{-2cm} = \left\{
\begin{array}{cc}
\ces_{\infty,r', \infty}(\omega_1,\omega_2,v) \bigcap L_{\infty}(\omega_3) \bigcap L_{\infty}(\omega_4) & ~ \mbox{if}  ~ \|u\|_{r,(0,\infty)} < \infty, \\
\ces_{\infty,r', \infty}(\omega_1,\omega_2,v) \bigcap L_{\infty}(\omega_3)  & ~ \mbox{if}  ~ \|u\|_{r,(0,\infty)} =\infty,
\end{array}
\right.
\end{align*}
where
$$
\omega_1:= A_{q,q}^*(w), \quad \omega_2 := A_{r,1}(u), \quad \omega_3(x) := v(x)\big\| A_{r',r'}^*(A_{r,1(u)})(t) \cdot \omega_1(t)\big\|_{\infty,(x,\infty)}, \quad \omega_4 := A_{q,q}^*(w)\cdot v.
$$
Moreover
\begin{align*}
\|f\|_{M(\cop_{r}(u), \ces_{p,q}(w,v))} \ap & \|f\|_{\ces_{\infty,r', \infty}(\omega_1,\omega_2,v)} + \|f\|_{\infty,\omega_3,(0,\infty)} +  \|u\|_{r,\I}^{-1} \|f\|_{\infty,\omega_4,(0,\infty)}.
\end{align*}
	
{\rm (ii)} If $q < r$, then
\begin{align*}
M(\cop_{r}(u), \ces_{p,q}(w,v)) & \\
& \hspace{-6cm} = \left\{
\begin{array}{cc}
\ces_{\infty,r', r \rw q}(\omega_1,\omega_2,v) \bigcap \ces_{\infty,\infty, r\rw q}(\omega_3,\omega_4,v) \bigcap L_{\infty}(\omega_5) & ~ \mbox{if}  ~ \|u\|_{r,(0,\infty)} < \infty, \\
\ces_{\infty,r', r \rw q}(\omega_1,\omega_2,v) \bigcap \ces_{\infty,\infty, r\rw q}(\omega_3,\omega_4,v)  & ~ \mbox{if}  ~ \|u\|_{r,(0,\infty)} =\infty,
\end{array}
\right.
\end{align*}
where
\begin{align*}
\omega_1:= A_{q,r}^*(w), ~~ &\omega_2 := A_{r,1}(u), ~~ \omega_3 := [A_{r',r'}^*(A_{r,1}(u))]^{\frac{r'}{q'}}\cdot[A_{r,1}(u)]^{\frac{r'}{r \rw q}},\\ &\omega_4 := A_{q,q}^*(w), ~~ \omega_5:= A_{q,q}^*(w) \cdot v. 
\end{align*}
Moreover
\begin{align*}
\|f\|_{M(\cop_{r}(u), \ces_{p,q}(w,v))} \ap & \|f\|_{\ces_{\infty,r', r \rw q}(\omega_1,\omega_2,v)} + \|f\|_{\ces_{\infty,\infty, r\rw q}(\omega_3,\omega_4,v)} + \|u\|_{r,\I}^{-1} \|f\|_{\infty,\omega_5,(0,\infty)}.
\end{align*}
\end{thm}

\begin{proof}
{\rm (i)} Let $r \le q$. Applying Lemma \ref{gluing.lem.3} with
$$
\b = r',~ a(t) =  V(t), ~ g(t) = V(t)^{r'} d \bigg( - \|u\|_{r,(0,t)}^{- r'}\bigg), ~ \mbox{and} ~ h (t) = \|w\|_{q,(t,\infty)} \quad (t > 0),
$$
by \cite[Theorem 2.11, (i)]{gmu}, we obtain that 
\begin{align*}
\|\Id\|_{\cop_r(u) \rw \ces_{p,q}(w,v)}  \ap & \, \sup_{x \in (0,\infty)} \bigg( \int_{(0,x)} V(t)^{r'} d \bigg( - \|u\|_{r,(0,t)}^{- r'}\bigg) \bigg)^{\frac{1}{r'}} \|w\|_{q,(x,\infty)} \\
& + \sup_{x \in (0,\infty)}  V(x) \bigg( \int_{(x,\infty)} d \bigg( - \|u\|_{r,(0,t)}^{- r'}\bigg) \bigg)^{\frac{1}{r'}} \|w\|_{q,(x,\infty)} \\
& + \|u\|_{r,\I}^{-1} \sup_{x\in \I} V(x)	\|w\|_{q,(x,\infty)}.
\end{align*}
Consequently, in view of \eqref{Fubini.1} and \eqref{Fubini.2}, we arrive at 
\begin{align*}
\|f\|_{M(\cop_{r}(u), \ces_{p,q}(w,v))} &\ap   \bigg\| \big\| \|f\|_{\infty, v, (0,\cdot)} \big\|_{r', \omega_2,(0,\cdot)} \bigg\|_{\infty,\omega_1,(0,\infty)} +  \|f\|_{\infty,\omega_3,(0,\infty)}  + \|u\|_{r,\I}^{-1} \| f \|_{\infty,\omega_4,(0,\infty)} \\
&=  \|f\|_{\ces_{\infty,r', \infty}(\omega_1,\omega_2,v)} +  \|f\|_{\infty,\omega_3,(0,\infty)}  + \|u\|_{r,\I}^{-1} \| f \|_{\infty,\omega_4,(0,\infty)}.
\end{align*}
	
{\rm (ii)} Let $q < r$. Applying Lemma \ref{gluing.lem} with
$$
\b =\frac{r \rw q}{ r'},~ a(t) =  V(t)^{r'}, ~ g(t) = V(t)^{r'} d \bigg( - \|u\|_{r,(0,t)}^{-r'}\bigg), ~ \mbox{and} ~ h (t) = \|w\|_{q,(t,\infty)}^{r'} \quad (t > 0),
$$
by \cite[Theorem 2.11, (ii)]{gmu}, we obtain that
\begin{align*}
\|\Id\|_{\cop_r(u) \rw \ces_{p,q}(w,v)}  & \\
& \hspace{-1.5cm} \ap \bigg(\int_{(0,\infty)} \bigg( \int_{(0,x)} V(t)^{r'} d \bigg( - \|u\|_{r,(0,t)}^{- r'}\bigg)\bigg)^{\frac{r \rw q}{q'}} \|w\|_{q,(x,\infty)}^{r \rw q} V(x)^{r'} d \bigg( - \|u\|_{r,(0,x)}^{- r'}\bigg)\bigg)^{\frac{1}{r \rw q}}\\
& \hspace{-1cm} + \bigg( \int_{(0,\infty)} \bigg( \sup_{t \in (0,x)} V(t)^{r'} \|w\|_{q,(t,\infty)}^{r'}\bigg)^{\frac{r \rw q}{r'}} \bigg( \int_{(x,\infty)} d \bigg( - \|u\|_{r,(0,t)}^{- r'}\bigg)\bigg)^{\frac{r \rw q}{q'}}  \, d \bigg( - \|u\|_{r,(0,x)}^{- r'}\bigg) \bigg)^{\frac{1}{r \rw q}}\\
& \hspace{-1cm} + \|u\|_{r,\I}^{-1} \sup_{t \in \I} V(t) \|w\|_{q,(t,\infty)}.
\end{align*}
Integrating by parts at the first integral on the right hand side, we get that
\begin{align*}
\|\Id\|_{\cop_r(u) \rw \ces_{p,q}(w,v)}  & \\
& \hspace{-1.5cm} \ap \bigg(\int_{(0,\infty)} \bigg( \int_{(0,x)} V(t)^{r'} d\bigg( - \|u\|_{r,(0,t)}^{- r'}\bigg)\bigg)^{\frac{r \rw q}{r'}}  d \bigg( - \|w\|_{q, (x,\infty)}^{r \rw q} \bigg) \bigg)^{\frac{1}{r \rw q}}	\\
& \hspace{-1cm} + \bigg( \int_{(0,\infty)} \bigg( \sup_{t \in (0,x)} V(t)^{r'} \|w\|_{q,(t,\infty)}^{r'}\bigg)^{\frac{r \rw q}{r'}} \bigg( \int_{(x,\infty)} d \bigg( - \|u\|_{r,(0,t)}^{- r'}\bigg)\bigg)^{\frac{r \rw q}{q'}}  \, d \bigg( - \|u\|_{r,(0,x)}^{- r'}\bigg) \bigg)^{\frac{1}{r \rw q}}\\
& \hspace{-1cm} + \|u\|_{r,\I}^{-1} \sup_{t \in \I} V(t) \|w\|_{q,(t,\infty)}.
\end{align*}
Therefore, in view of \eqref{Fubini.1}, we arrive at 
\begin{align*}
\|f\|_{M(\cop_{r}(u), \ces_{p,q}(w,v))} &\ap \bigg\| \big\| \|f\|_{\infty, v, (0,\cdot)} \big\|_{r', \omega_2,(0,\cdot)} \bigg\|_{r \rw q,\omega_1,(0,\infty)}  + \bigg\| \big\| \|f\|_{\infty, v, (0,\cdot)} \big\|_{\infty, \omega_4,(0,\cdot)} \bigg\|_{r\rw q,\omega_3,(0,\infty)}\\
&\hspace{1cm} + \|u\|_{r,\I}^{-1} \|f\|_{\infty,\omega_5} \\
& =  \|f\|_{\ces_{\infty,r', r \rw q}(\omega_1,\omega_2,v)} + \|f\|_{\ces_{\infty,\infty, r\rw q}(\omega_3,\omega_4,v)} + \|u\|_{r,\I}^{-1} \|f\|_{\infty,\omega_5,(0,\infty)}.
\end{align*}
\end{proof}

\

Amiran Gogatishvili \\
Institute of Mathematics, Academy of Sciences of the Czech Republic,    \v Zitn\'a~25,  115~67 Praha~1, Czech Republic \\
E-mail: gogatish@math.cas.cz \\

Rza Mustafayev\\
Department of Mathematics, Faculty of Science, Karamanoglu Mehmetbey University, Karaman, 70100, Turkey \\
E-mail: rzamustafayev@gmail.com\\

Tugce {\"U}nver \\
Department of Mathematics, Faculty of Science and Arts, Kirikkale
University, 71450 Yahsihan, Kirikkale, Turkey\\
E-mail: tugceunver@gmail.com \\

\end{document}